\documentclass[11pt]{article} 
\usepackage{amsfonts,amsmath,latexsym,amssymb,mathrsfs,amsthm,comment}
\usepackage{graphicx}
\usepackage{xcolor}
\usepackage{booktabs}

\let\OLDthebibliography\thebibliography
\renewcommand\thebibliography[1]{
  \OLDthebibliography{#1}
  \setlength{\parskip}{1pt}
  \setlength{\itemsep}{0pt plus 0.0ex}
}

\def\numberlikeadb{\global\def\theequation{\thesection.\arabic{equation}}}
\numberlikeadb
\newtheorem{theorem}{Theorem}[section]
\newtheorem{lemma}[theorem]{Lemma}
\newtheorem{corollary}[theorem]{Corollary}

\newtheorem{remark}[theorem]{Remark}

\usepackage{lscape}
\usepackage{caption}
\usepackage{multirow}
\allowdisplaybreaks
\begin{document}

\title{Integrals of products of four modified Bessel functions}
\author{Robert E. Gaunt\footnote{Department of Mathematics, The University of Manchester, Oxford Road, Manchester M13 9PL, UK, robert.gaunt@manchester.ac.uk}}

\date{} 
\maketitle

\vspace{-5mm}

\begin{abstract} We evaluate definite integrals involving the product of four modified Bessel functions of the first and second kind and a power function. We provide general formulas expressed in terms of the Meijer $G$-function and generalized hypergeometric and Lauricella $F_C$ functions, and study a number of special cases in which the integrals can be evaluated in terms of simpler special functions or indeed take an elementary form. As a consequence, we deduce some new formulas for definite integrals of products of four Airy functions.
\end{abstract}



\section{Introduction}

Definite integrals involving products of Bessel functions have received much attention in the literature. General formulas for products of two and three Bessel functions are compiled in standard references such as \cite{e54,g07,olver,integralbook}. General formulas for the product of four or more Bessel functions are harder to come by, but this topic has also received much interest (see, for example \cite{b07,b36,c09,f13,g12,m13,m94,ps21,s64,z17}), and some of these formulas are tabulated in the standard references \cite{e54,g07,integralbook}. Evaluation of the (modified) Bessel moments $\int_0^\infty x^kI_\alpha^m(x)K_\beta^n(x)\,\mathrm{d}x$, $\alpha,\beta\in\{0,1\}$, $k,m,n\in\mathbb{N}$, has received considerable interest on account of the fact that these integrals arise in areas of mathematical physics, including condensed matter physics \cite{b07,o05} and quantum field theory \cite{b08,b21}. Here, $I_\nu(x)$ and $K_\nu(x)$ denote the modified Bessel functions of the first and second kind, respectively.

For integrals of products of four or more modified Bessel functions $I_\nu(x)$ and $K_\nu(x)$, whilst the case $\nu\in\{0,1\}$ has received much attention (for example in the aforementioned references from the physics literature), the general case $\nu\in\mathbb{C}$ (such that the integral exists) appears, to the best knowledge of this author, to have not received a detailed treatment in the literature. In this paper, we fill in this gap by evaluating general definite integrals involving products of four modified Bessel functions of the first and second kind, that is integrals of the form $\int_0^\infty x^{s-1}\{\prod_{i=1}^{4-n} I_{\mu_i}(a
_i x)\}\{\prod_{j=1}^n K_{\nu_j}(b_j x)\}\,\mathrm{d}x$ for $n=1,2,3,4$, with the parameters being such that the integral converges. In Section \ref{sec2.1}, we provide general formulas expressed in terms of the Meijer $G$-function, the generalized hypergeometric function and the Lauricella $F_C$ function, all of which are defined in Appendix \ref{appa}. For certain parameter values, an application of reduction formulas for the Meijer $G$-function and the generalized hypergeometric function allows for evaluation of the integrals in terms of simpler special functions. We study several such cases in Section \ref{sec2.2}, obtaining a number of rather elegant expressions, and identify certain cases in which the integrals take an elementary form. As far as this author is aware, the majority of the integral formulas presented in Section \ref{sec2} are new; we explicitly point out instances in which integral formulas have previously been derived in the literature.

Definite integrals involving products of Airy functions also arise in mathematical physics \cite{m78,t15,vs10} and the evaluation of these integrals has been the subject of a number of works; see, for example, \cite{ar15,am15,l93,r95,r97a,r97b,vs97,v10}. In Section \ref{sec3}, we apply some of our formulas from Section \ref{sec2} to obtain new formulas for definite integrals of products of four Airy functions. Our formulas complement those of \cite{ar15,l93,r97b} in which closed-form formulas were given for definite integrals involving products of four Airy functions. The proofs of the results from Section \ref{sec2} are given in Section \ref{sec4}. Finally, in Appendix \ref{appa}, we state a number of basic properties of special functions that are utilised in this paper.



\section{Integrals of products of four modified Bessel functions}\label{sec2}

\subsection{General formulas}\label{sec2.1}

In the following theorem, we evaluate the Mellin transform of the product of four modified Bessel functions. The formulas are expressed in terms of the Meijer $G$-function, the regularized generalized hypergeometric function and Lauricella $F_C$ function, all of which are defined in Appendix \ref{appa}. The proof of Theorem \ref{thm1.1} is deferred to Section \ref{sec4}, as is the case for the proofs of all theorems in Section \ref{sec2}. 

\begin{theorem}\label{thm1.1} 1. Suppose that $\mathrm{Re}\,a,\mathrm{Re}\,b>0$ and $\mathrm{Re}\, s>|\mathrm{Re}\,\alpha|+|\mathrm{Re}\,\beta|+|\mathrm{Re}\,\gamma|+|\mathrm{Re}\,\delta|$. Then
\begin{align} 
&\int_0^\infty x^{s-1} K_\alpha(a x) K_\beta(a x) K_\gamma(b x) K_\delta(b x) \,\mathrm{d}x\nonumber\\
&\quad=\frac{\pi}{8a^{s}}G_{6,6}^{4,4}\bigg(\frac{a^2}{b^2}\;\bigg|\;{\frac{2-\gamma-\delta}{2},\frac{2+\gamma-\delta}{2},\frac{2+\gamma-\delta}{2},\frac{2+\gamma+\delta}{2},\frac{s}{2},\frac{s+1}{2} \atop \frac{s+\alpha+\beta}{2},\frac{s-\alpha+\beta}{2},\frac{s+\alpha-\beta}{2},\frac{s-\alpha-\beta}{2},\frac{1}{2},1}\bigg).   \label{for1}
\end{align} 

\noindent 2. Suppose that $\mathrm{Re}\,a,\mathrm{Re}\,b>0$ and $\mathrm{Re}\,s>|\mathrm{Re}\,\beta|+|\mathrm{Re}\,\gamma|+|\mathrm{Re}\,\delta|-\mathrm{Re}\,\alpha$. Then
\begin{align} 
&\int_0^\infty x^{s-1} I_\alpha(a x) K_\beta(a x) K_\gamma(b x) K_\delta(b x) \,\mathrm{d}x\nonumber\\
&\quad=\frac{1}{8a^s}G_{6,6}^{2,6}\bigg(\frac{a^2}{b^2}\;\bigg|\;{\frac{2-\gamma-\delta}{2},\frac{2+\gamma-\delta}{2},\frac{2+\gamma-\delta}{2},\frac{2+\gamma+\delta}{2},\frac{s}{2},\frac{s+1}{2} \atop \frac{s+\alpha+\beta}{2},\frac{s+\alpha-\beta}{2},\frac{s-\alpha+\beta}{2},\frac{s-\alpha-\beta}{2},\frac{1}{2},1}\bigg).   \label{k3igen}
\end{align}  

\noindent 3. Suppose that $\mathrm{Re}\,b>\mathrm{Re}\,a>0$ and $\mathrm{Re}\,s>|\mathrm{Re}\,\gamma|+|\mathrm{Re}\,\delta|-\mathrm{Re}(\alpha+\beta)$, or $\mathrm{Re}\,b=\mathrm{Re}\,a>0$ and $|\mathrm{Re}\,\gamma|+|\mathrm{Re}\,\delta|-\mathrm{Re}(\alpha+\beta)<s<2$. Then
\begin{align}
&\int_0^\infty x^{s-1} I_\alpha(a x) I_\beta(a x) K_\gamma(b x) K_\delta(b x) \,\mathrm{d}x\nonumber\\
&\quad= \frac{1}{4a^s}\bigg(\frac{a}{b}\bigg)^{\lambda+s}\Gamma\big(\tfrac{\lambda+1}{2}\big)\Gamma\big(\tfrac{\lambda+2}{2}\big)\Gamma\big(\tfrac{s+\lambda+\gamma+\delta}{2}\big)\Gamma\big(\tfrac{s+\lambda-\gamma+\delta}{2}\big)\Gamma\big(\tfrac{s+\lambda+\gamma-\delta}{2}\big)\Gamma\big(\tfrac{s+\lambda-\gamma-\delta}{2}\big)\nonumber\\
&\quad\quad\times {}_6\tilde{F}_5\bigg({\frac{\lambda+1}{2},\frac{\lambda+2}{2},\frac{s+\lambda+\gamma+\delta}{2},\frac{s+\lambda-\gamma+\delta}{2},\frac{s+\lambda+\gamma-\delta}{2},\frac{s+\lambda-\gamma-\delta}{2} \atop 
\alpha+1,\beta+1,\lambda+1,\frac{s+\lambda}{2},\frac{s+\lambda+1}{2}} \;\bigg|\;\frac{a^2}{b^2}\bigg), \label{iikkgen}
\end{align}
where $\lambda=\alpha+\beta$.

\vspace{3mm}

\noindent 4. Let $\mathrm{Re}\,a,\mathrm{Re}\,b,\mathrm{Re}\,c,\mathrm{Re}\,d>0$. Suppose that $\mathrm{Re}\,d>\mathrm{Re}(a+b+c)$ and $\mathrm{Re}\,s>|\mathrm{Re}\,\delta|-\mathrm{Re}(\alpha+\beta+\gamma)$, or $\mathrm{Re}\,d=\mathrm{Re}(a+b+c)$ and $|\mathrm{Re}\,\delta|-\mathrm{Re}(\alpha+\beta+\gamma)<s<2$. Then
\begin{align}
&\int_0^\infty x^{s-1} I_\alpha(a x) I_\beta(b x) I_\gamma(c x) K_\delta(d x) \,\mathrm{d}x=\frac{2^{s-2}}{d^2}\bigg(\frac{a^2}{d^2}\bigg)^\alpha\bigg(\frac{b^2}{d^2}\bigg)^\beta\bigg(\frac{c^2}{d^2}\bigg)^\gamma\times \nonumber \\
&\quad\times F_C^{(3)}\bigg(\frac{s+\alpha+\beta+\gamma+\delta}{2},\frac{s+\alpha+\beta+\gamma-\delta}{2};\alpha+1,\beta+1,\gamma+1;\frac{a^2}{d^2},\frac{b^2}{d^2},\frac{c^2}{d^2}\bigg).\label{lion5}
\end{align}
\end{theorem}

\begin{remark}
The conditions on the parameters $s$, $a,b$ and $\alpha,\beta,\gamma,\delta$ under which the integral formulas of Theorem \ref{thm1.1} hold can be readily appreciated by consideration of the limiting behaviour of the modified Bessel functions $I_\nu(z)$ and $K_\nu(z)$ in the limits $z\rightarrow0$ and $z\rightarrow\infty$, as given in the limiting forms (\ref{Itend0})--(\ref{Ktendinfinity}) and the subsequent discussion in Appendix \ref{appa1}. A similar comment applies to all forthcoming integral formulas.  
\end{remark}

The integral formula (\ref{lion5}) is in fact a special case of formula (\ref{thm1.2for}), as given in the following theorem.

\begin{theorem}\label{thm1.2} Let $n\geq1$ and suppose $\mathrm{Re}\,a_k>0$ for $k=1,\ldots,n$. Suppose further that $\mathrm{Re}\,b>\sum_{k=1}^n\mathrm{Re}\,a_k$ and $\mathrm{Re}\,s>|\mathrm{Re}\,b|-\sum_{k=1}^n\mathrm{Re}\,\alpha_k$, or $\mathrm{Re}\,b=\sum_{k=1}^n\mathrm{Re}\,a_k$ and $|\mathrm{Re}\,b|-\sum_{k=1}^n\mathrm{Re}\,\alpha_k<s<2$. Then
\begin{align}
&\int_0^\infty x^{s-1}\bigg\{\prod_{k=1}^n I_{\alpha_k}(a_kx)\bigg\}K_\beta(bx)\,\mathrm{d}x=\frac{2^{s-2}}{b^2}\bigg\{\prod_{k=1}^n\bigg(\frac{a_k^2}{b^2}\bigg)^{\alpha_k}\bigg\}\times \nonumber \\
&\,\,\times F_C^{(n)}\bigg(\frac{s+\alpha_1+\cdots+\alpha_n+\beta}{2},\frac{s+\alpha_1+\cdots+\alpha_n-\beta}{2};\alpha_1+1,\ldots,\alpha_n+1;\frac{a_1^2}{b^2},\cdots,\frac{a_n^2}{b^2}\bigg). \label{thm1.2for}
\end{align}    
\end{theorem}

For certain parameter values, the integral formulas given in parts 1, 2 and 3 of Theorem \ref{thm1.1} reduce to Meijer $G$-functions and generalized hypergeometric functions of lower order. Some examples are given in the following corollary.

\begin{corollary}\label{cor2.3} 1. Suppose that $\mathrm{Re}\,a,\mathrm{Re}\,b>0$ and $\mathrm{Re}\,s>2|\mathrm{Re}\,\alpha|+2|\mathrm{Re}\,\beta|$. Then
\begin{align} \label{corf1}
\int_0^\infty x^{s-1} K_\alpha^2(a x) K_\beta^2(b x) \,\mathrm{d}x=\frac{\pi}{8a^{s}}G_{4,4}^{3,3}\bigg(\frac{a^2}{b^2}\;\bigg|\;{1+\beta,1-\beta,1,\frac{s+1}{2} \atop \frac{s}{2}+\alpha,\frac{s}{2}-\alpha,\frac{s}{2},\frac{1}{2}}\bigg).
\end{align}
2. Suppose that $\mathrm{Re}\,a,\mathrm{Re}\,b>0$ and $\mathrm{Re}\,s>2|\mathrm{Re}\,\beta|+|\mathrm{Re}\,\alpha|
-\mathrm{Re}\,\alpha$. Then
\begin{align} \label{corf2}
\int_0^\infty x^{s-1} I_\alpha(ax)K_\alpha(ax) K_\beta^2(bx)\,\mathrm{d}x=\frac{1}{8a^{s}}G_{4,4}^{2,4}\bigg(\frac{a^2}{b^2}\;\bigg|\;{1+\beta,1-\beta,1,\frac{s+1}{2} \atop \frac{s}{2}+\alpha,\frac{s}{2},\frac{s}{2}-\alpha,\frac{1}{2}}\bigg). 
\end{align}
3. Suppose that $\mathrm{Re}\,b>\mathrm{Re}\,a>0$ and $\mathrm{Re}\,s>2|\mathrm{Re}\,\beta|-2\mathrm{Re}\,\alpha$, or $\mathrm{Re}\,b=\mathrm{Re}\,a>0$ and $2|\mathrm{Re}\,\beta|-2\mathrm{Re}\,\alpha<s<2$. Then
\begin{align}
\int_0^\infty x^{s-1} I_\alpha^2(a x) K_\beta^2(b x) \,\mathrm{d}x
&= \frac{1}{4a^s}\bigg(\frac{a}{b}\bigg)^{2\alpha+s}\Gamma\bigg(\alpha+\frac{1}{2}\bigg)\Gamma\bigg(\alpha+\beta+\frac{s}{2}\bigg)\Gamma\bigg(\alpha+\frac{s}{2}\bigg)\Gamma\bigg(\alpha-\beta+\frac{s}{2}\bigg)\nonumber\\
&\quad\times {}_4\tilde{F}_3\bigg({\alpha+\frac{1}{2},\alpha+\beta+\frac{s}{2},\alpha+\frac{s}{2},\alpha-\beta+\frac{s}{2} \atop \alpha+1,2\alpha+1,\alpha+\frac{s+1}{2}}\;\bigg|\;\frac{a^2}{b^2}\bigg).  \label{corf3}
\end{align}
\end{corollary}

\begin{proof}
With the given parameter values, applying the reductions formulas (\ref{lukeformula0}) and (\ref{lukeformula}) for the Meijer $G$-function to the integral formulas (\ref{for1}) and (\ref{k3igen}) yields formulas (\ref{corf1}) and (\ref{corf2}), respectively. With the given parameter values, formula (\ref{corf3}) follows from applying the reduction formula (\ref{fred2}) to the integral formula (\ref{iikkgen}). 
\end{proof}

\subsection{Reduction to simpler special functions and elementary forms}\label{sec2.2}

In this section, we study in detail several cases in which the integral formulas of Section \ref{sec2.1} take a simpler form, with the evaluation of the integrals expressed in terms of generalized hypergeometric functions or simpler special functions. We also note a number of cases in which the integrals take an elementary form. 

Some of our integral formulas are stated in terms of the digamma function, which is defined, for $z\in\mathbb{C}$, by $\psi(z)=\Gamma'(z)/\Gamma(z)$. Some integral formulas are also expressed in terms of polylogarithms, which are defined, for $s\in\mathbb{C}$ and $z\in\mathbb{C}$, by $\mathrm{Li}_s(z)=\sum_{k=1}^\infty z^k/k^s$, where for fixed $s$ such that $\mathrm{Re}\,s>1$, the series is convergent for $|z|\leq1$, and by analytic continuation elsewhere. The functions $\mathrm{Li}_2(z)$ and $\mathrm{Li}_3(z)$ are referred to as the \emph{dilogarithm} and \emph{trilogarithm}. Henceforth, we will denote the Euler–Mascheroni constant by $\gamma=0.57721\ldots$.

\begin{theorem}\label{cor1.3} Let $\alpha\in\mathbb{C}\setminus\{0\}$ and suppose that $\mathrm{Re}\,a,\mathrm{Re}\,b>0$.

\vspace{3mm}

\noindent 1. Let $|\mathrm{Re}\,\alpha|<1/4$. Then
\begin{align} 
&\int_0^\infty K_\alpha^2(a x) K_\alpha^2(b x) \,\mathrm{d}x \nonumber \\
&\quad=\frac{\pi^4}{8b}\csc(2\pi\alpha)\bigg\{\frac{\csc(\pi\alpha)\Gamma(2\alpha+\frac{1}{2})}{\Gamma(\frac{1}{2}-\alpha)\Gamma^3(1+\alpha)}\bigg(\frac{a}{2b}\bigg)^{2\alpha} {}_4F_3\bigg({\frac{1}{2},\alpha+\frac{1}{2},\alpha+\frac{1}{2},2\alpha+\frac{1}{2} \atop 2\alpha+1,\alpha+1,\alpha+1}\;\bigg|\;\frac{a^2}{b^2}\bigg)\nonumber\\
&\quad\quad+\frac{\csc(\pi\alpha)\Gamma(\frac{1}{2}-2\alpha)}{\Gamma(\alpha+\frac{1}{2})\Gamma^3(1-\alpha)}\bigg(\frac{a}{2b}\bigg)^{-2\alpha}{}_4F_3\bigg({\frac{1}{2},\frac{1}{2}-2\alpha,\frac{1}{2}-\alpha,\frac{1}{2}-\alpha \atop 1-2\alpha,1-\alpha,1-\alpha}\;\bigg|\;\frac{a^2}{b^2}\bigg)\nonumber\\
&\quad\quad-\frac{2}{\pi\alpha}{}_4F_3\bigg({\frac{1}{2},\frac{1}{2},\alpha+\frac{1}{2},\frac{1}{2}-\alpha \atop 1,1-\alpha,\alpha+1}\;\bigg|\;\frac{a^2}{b^2}\bigg)\bigg\}.\label{slv1}
\end{align} 
\noindent 2. Suppose that $|\mathrm{Re}\,\alpha|<1/2$. If $a\not=b$ then
\begin{align}
\int_0^\infty xK_\alpha^2(a x) K_\alpha^2(b x) \,\mathrm{d}x&=\frac{\pi^2}{8a^2}\csc^2(\pi\alpha)\bigg\{\frac{(a/b)^{2+2\alpha}}{1+2\alpha}\,{}_2F_1\bigg({1,\frac{1}{2}+\alpha \atop \frac{3}{2}+\alpha}\;\bigg|\;\frac{a^2}{b^2}\bigg)\nonumber\\
&\quad+\frac{(a/b)^{2-2\alpha}}{1-2\alpha}\,{}_2F_1\bigg({1,\frac{1}{2}-\alpha \atop \frac{3}{2}-\alpha} \;\bigg|\;\frac{a^2}{b^2}\bigg)-\frac{2a}{b}\tanh^{-1}\bigg(\frac{a}{b}\bigg)\bigg\}, \label{forab}
\end{align} 
and if $a=b$ then
\begin{align}
\int_0^\infty xK_\alpha^4(a x) \,\mathrm{d}x=-\frac{\pi^2}{16a^2}\mathrm{csc}^2(\pi\alpha)\bigg\{\psi\bigg(\frac{1}{2}-\alpha\bigg)+\psi\bigg(\frac{1}{2}+\alpha\bigg)+4\ln(2)+2\gamma\bigg\}. \label{foraa}
\end{align}
Now suppose that $\alpha=0$. Without loss of generality, suppose $|a/b|\leq1$. Then, for $a\not=b$,
\begin{align} \label{li2}
\int_0^\infty xK_0^2(a x) K_0^2(b x) \,\mathrm{d}x&=\frac{1}{2ab}\bigg\{\ln^2\bigg(\frac{a}{b}\bigg)\tanh^{-1}\bigg(\frac{a}{b}\bigg)-\ln\bigg(\frac{a}{b}\bigg)\bigg[\mathrm{Li}_2\bigg(\frac{a}{b}\bigg)-\mathrm{Li}_2\bigg(-\frac{a}{b}\bigg)\bigg]\nonumber\\
&\quad+\mathrm{Li}_3\bigg(\frac{a}{b}\bigg)-\mathrm{Li}_3\bigg(-\frac{a}{b}\bigg)\bigg\},  
\end{align}
and, for $a=b$,
\begin{align}\label{fox1}
\int_0^\infty xK_0^4(a x) \,\mathrm{d}x=\frac{7}{8a^2}\zeta(3),
\end{align} 
where $\zeta(3)=\sum_{k=1}^\infty 1/k^3$, a special value of the Riemann zeta function.
\end{theorem}

\begin{remark} 1. The integral formula (\ref{slv1}) is not valid when $\alpha=0$. However, when $\alpha=0$ and $a=b$ the integral can be expressed as a single generalized hypergeometric function:
\begin{align*}
\int_0^\infty K_0^4(a x) \,\mathrm{d}x =\frac{\pi^4}{4a}\,{}_4F_3\bigg({\frac{1}{2},\frac{1}{2},\frac{1}{2},\frac{1}{2} \atop 1,1,1}\; \bigg| \;1\bigg)   
\end{align*}    
(see \cite{b08}). The integral formula (\ref{fox1}) was also earlier derived by \cite{b08}.

\vspace{3mm}

\noindent 2. For certain values of $\alpha$ the integral formulas (\ref{forab}) and (\ref{foraa}) reduce to elementary forms. Applying the reduction formulas (\ref{red11}) and (\ref{red22}) to the integral formula (\ref{forab}) and simplifying the resulting expression by using the basic identify
\begin{align*}
\ln(z^{1/3}+z^{1/6}+1)-\ln(z^{1/3}-z^{1/6}+1)=2\tanh^{-1}\bigg(\frac{z^{1/6}}{1+z^{1/3}}\bigg)   
\end{align*}
yields the elementary evaluation
\begin{align}
\int_0^\infty xK_{1/3}^2(a x) K_{1/3}^2(b x) \,\mathrm{d}x&=\frac{\pi^2}{6ab}\bigg\{2\tanh^{-1}\bigg(\frac{a^{1/3}}{b^{1/3}}\bigg)-2\tanh^{-1}\bigg(\frac{a}{b}\bigg) \nonumber \\
&\quad+\tanh^{-1}\bigg(\frac{(ab)^{1/3}}{a^{2/3}+b^{2/3}}\bigg)\bigg\}. \nonumber
\end{align}
Applying the addition formula
\begin{align}\label{addtanh}
\tanh^{-1}(u)\pm\tanh^{-1}(v)=\tanh^{-1}\bigg(\frac{u\pm v}{1\pm uv}\bigg)    
\end{align}
(see \cite[equation 4.38.17]{olver}) yields a simpler expression in terms of a single inverse hyperbolic tangent:
\begin{align}
\int_0^\infty xK_{1/3}^2(a x) K_{1/3}^2(b x) \,\mathrm{d}x=\frac{\pi^2}{2ab}\tanh^{-1}\bigg(\frac{(ab)^{1/3}}{a^{2/3}+b^{2/3}}\bigg).\label{kab13} 
\end{align}
Similarly, applying the reduction formulas (\ref{red11}) and (\ref{red22}) to the integral formula (\ref{forab}) yields the elementary expression
\begin{align}
\int_0^\infty xK_{1/4}^2(a x) K_{1/4}^2(b x) \,\mathrm{d}x&=\frac{\pi^2}{2ab}\bigg\{\tanh^{-1}\bigg(\sqrt{\frac{a}{b}}\,\bigg)-\tanh^{-1}\bigg(\frac{a}{b}\bigg)\bigg\}\nonumber  \\ 
&=\frac{\pi^2}{2ab}\coth^{-1}\bigg(\sqrt{\frac{a}{b}}+1+\sqrt{\frac{b}{a}}\bigg), \nonumber
\end{align} 
where the second equality was obtained by again using the addition formula (\ref{addtanh}) and the basic relation $\tanh^{-1}(1/z)=\coth^{-1}(z)$. 

\vspace{3mm}

\noindent 3. 
We now consider some further elementary integral formulas that can be deduced from the integral formula (\ref{foraa}). To this end, we note the following special values of the digamma function. If $p,q$ are integers such that $0<p<q$, then
\begin{align}\label{psif}
\psi\bigg(\frac{p}{q}\bigg)=-\gamma-\ln(q)-\frac{\pi}{2}\cot\bigg(\frac{\pi p}{q}\bigg)+\frac{1}{2}\sum_{k=1}^{q-1}\cos\bigg(\frac{2\pi kp}{q}\bigg)\ln\bigg(2-2\cos\bigg(\frac{2\pi k}{q}\bigg)\bigg)   
\end{align}
(see \cite[equation 5.4.19]{olver}). Let $m\geq3$ be a positive integer. Then
\begin{align}
\psi\bigg(\frac{1}{2}-\frac{1}{m}\bigg)+\psi\bigg(\frac{1}{2}+\frac{1}{m}\bigg)&= \psi\bigg(\frac{m-2}{2m}\bigg)+\psi\bigg(\frac{m+2}{2m}\bigg)\nonumber \\
&=-2\gamma-2\ln(2m)+\sum_{k=1}^{q-1}(-1)^k\cos\bigg(\frac{2\pi k}{m}\bigg)\ln\bigg(2-2\cos\bigg(\frac{\pi k}{m}\bigg)\bigg), \label{psi2}
\end{align}
where in obtaining the second equality we applied formula (\ref{psif}) and the basic trigonometric identities 
\[\cot\bigg(\frac{\pi(m-2)}{2m}\bigg)+\cot\bigg(\frac{\pi(m+2)}{2m}\bigg)=0\]
and
\begin{align*}
\cos\bigg(\frac{\pi k(m-2)}{m}\bigg)+ \cos\bigg(\frac{\pi k(m+2)}{m}\bigg)=2\cos(\pi k)\cos\bigg(\frac{2\pi k}{m}\bigg)=2(-1)^k\cos\bigg(\frac{2\pi k}{m}\bigg), \quad k\in\mathbb{Z}. 
\end{align*}

Applying formula (\ref{psi2}) to the integral formula (\ref{foraa}) and calculating the trigonometric functions and making use of the basic formula $\coth^{-1}(x)=(1/2)\ln((x+1)/(x-1))$, for $x>1$, yields the following elementary integral formulas:
\begin{align}
\int_0^\infty xK_{1/3}^4(a x) \,\mathrm{d}x&=\frac{\pi^2}{4a^2}\ln(3),\nonumber
\\
\int_0^\infty xK_{1/4}^4(a x) \,\mathrm{d}x&=\frac{\pi^2}{4a^2}\ln(2),\nonumber
\\
\int_0^\infty xK_{1/6}^4(a x) \,\mathrm{d}x&=\frac{\pi^2}{4a^2}\ln\bigg(\frac{27}{16}\bigg),\nonumber \\
\int_0^\infty xK_{1/8}^4(a x) \,\mathrm{d}x&=\frac{\pi^2}{4a^2}(2+\sqrt{2})\big(2\ln(2)-\sqrt{2}\coth^{-1}(\sqrt{2})\big),\nonumber\\
\int_0^\infty xK_{1/10}^4(a x) \,\mathrm{d}x&=\frac{\pi^2}{16a^2}(3+\sqrt{5})\bigg(\ln\bigg(\frac{3125}{256}\bigg)-2\sqrt{5}\coth^{-1}(\sqrt{5})\bigg),\nonumber\\
\int_0^\infty xK_{1/12}^4(a x) \,\mathrm{d}x&=\frac{\pi^2}{4a^2}(2+\sqrt{3})\big(\ln(108)-4\sqrt{3}\coth^{-1}(\sqrt{3})\big).\nonumber
\end{align}
\end{remark}

\begin{theorem}\label{cor2.5} Let $\alpha\in\mathbb{C}\setminus\{0\}$ and suppose that $\mathrm{Re}\,a,\mathrm{Re}\,b>0$.

\vspace{3mm}

\noindent 1. Let $-1/4<\mathrm{Re}\,\alpha<1/2$. Then
\begin{align}
&\int_0^\infty I_\alpha(a x)K_\alpha(a x) K_\alpha^2(b x) \,\mathrm{d}x \nonumber \\
&\quad=\frac{\pi^2}{8b}\bigg\{\frac{\sec(\pi\alpha)}{\alpha}\, {}_4F_3\bigg({\frac{1}{2},\frac{1}{2},\alpha+\frac{1}{2},\frac{1}{2}-\alpha \atop 1,\alpha+1,1-\alpha}\;\bigg|\;\frac{a^2}{b^2}\bigg)\nonumber\\
&\quad\quad-\frac{\Gamma(\alpha+\frac{1}{2})\Gamma(2\alpha+\frac{1}{2})\csc(\pi\alpha)}{\Gamma^3(\alpha+1)}\bigg(\frac{a}{2b}\bigg)^{\alpha} {}_4F_3\bigg({\frac{1}{2},\alpha+\frac{1}{2},\alpha+\frac{1}{2},2\alpha+\frac{1}{2} \atop \alpha+1,\alpha+1,2\alpha+1}\;\bigg|\;\frac{a^2}{b^2}\bigg)\bigg\}. \label{sch1}    \end{align}
\noindent 2. Suppose that $-1/2<\mathrm{Re}\,\alpha<1$. Then, for $a\not=b$, 
\begin{align}
\int_0^\infty x I_\alpha(ax)K_\alpha(ax)K_\alpha^2(bx)\,\mathrm{d}x=\frac{\pi\csc(\pi\alpha)}{4b^2}\bigg\{\frac{b}{a}\tanh^{-1}\bigg(\frac{a}{b}\bigg)-\frac{(a/b)^{2\alpha}}{1+2\alpha}\,{}_2F_1\bigg({1,\frac{1}{2}+\alpha \atop \frac{3}{2}+\alpha}\;\bigg|\;\frac{a^2}{b^2}\bigg)\bigg\}, \label{forab2}
\end{align}
and if $a=b$,
\begin{align}
\int_0^\infty xI_\alpha(ax)K_\alpha^3(ax)\,\mathrm{d}x=\frac{\pi}{8a^2}\csc(\pi\alpha)\bigg\{\psi\bigg(\frac{1}{2}+\alpha\bigg)+2\ln(2)+\gamma\bigg\}. \label{foraa2}   
\end{align}
When $\alpha=0$ we have, for $|a/b|\leq1$ and $a\not=b$,
\begin{align}\label{lili}
\int_0^\infty x I_0(ax)K_0(ax)K_0^2(bx)\,\mathrm{d}x=\frac{1}{4ab}\bigg\{\mathrm{Li}_2\bigg(\frac{a}{b}\bigg)-\mathrm{Li}_2\bigg(-\frac{a}{b}\bigg)-2\ln\bigg(\frac{a}{b}\bigg)\tanh^{-1}\bigg(\frac{a}{b}\bigg)\bigg\}, 
\end{align}
for $|a/b|\geq1$ and $a\not=b$,
\begin{align}\label{lilix}
\int_0^\infty x I_0(ax)K_0(ax)K_0^2(bx)\,\mathrm{d}x=\frac{1}{4ab}\bigg\{\frac{\pi^2}{2}-\mathrm{Li}_2\bigg(\frac{b}{a}\bigg)+\mathrm{Li}_2\bigg(-\frac{b}{a}\bigg)+2\ln\bigg(\frac{b}{a}\bigg)\tanh^{-1}\bigg(\frac{b}{a}\bigg)\bigg\}, 
\end{align}
and, for $a=b$,
\begin{align}\label{fox2}
\int_0^\infty xK_0^3(ax)I_0(ax)\,\mathrm{d}x=\frac{\pi^2}{16a^2}.    
\end{align}
\end{theorem}

\begin{remark} 1. The integral formula (\ref{sch1}) is not valid when $\alpha=0$. However, in the case $\alpha=0$ and $a=b$, \cite{b08} derived the following infinite series representation for the integral:
\begin{align*}
\int_0^\infty I_0(at)K_0^3(at)\,\mathrm{d}t=\frac{\pi^2}{2a}\sum_{k=0}^\infty\binom{2n}{n}\frac{1}{2^{8n}}\bigg\{\frac{1}{4n+1}+\sum_{j=0}^{2n-1}\frac{2}{2j+1}\bigg\}.    
\end{align*}
The integral formula (\ref{fox2}) was also previously derived by \cite{b08}. 

\vspace{3mm}

\noindent 2. We note that the integral formulas (\ref{lili}) and (\ref{lilix})
actually hold for all $a\not=b$ such that $\mathrm{Re}\,a,\mathrm{Re}\,b>0$, due to the identity
\begin{align}
&\mathrm{Li}_2(z)-\mathrm{Li}_2(-z)-2\ln(z)\tanh^{-1}(z)\nonumber\\
&\quad=\frac{\pi^2}{2}-\mathrm{Li}_2\bigg(\frac{1}{z}\bigg)+\mathrm{Li}_2\bigg(-\frac{1}{z}\bigg)+2\ln\bigg(\frac{1}{z}\bigg)\tanh^{-1}\bigg(\frac{1}{z}\bigg), \quad \mathrm{Re}\,z>0, \label{jx}  
\end{align}
which tells us that the right-hand sides of equations (\ref{lili}) and (\ref{lilix}) are equal for all $a\not=b$ such that $\mathrm{Re}\,a,\mathrm{Re}\,b>0$. The equality (\ref{jx}) is readily verified as follows. Denote the LHS and RHS of (\ref{jx}) by $L(z)$ and $R(z)$, respectively. Then, using the differentiation formula $\frac{\mathrm{d}}{\mathrm{d}z}(\mathrm{Li}_2(z))=-\ln(1-z)/z$ (see \cite[equation 25.12.2]{olver}) and the standard derivative $\tanh^{-1}(z)=1/(1-z^2)$ one obtains that
\begin{align*}
L'(z)=R'(z)=\frac{2\ln(z)}{z^2-1}.    
\end{align*}
Now, let $F(z):=L(z)-R(z)$. Then $F'(z)=0$, so that $F(z)$ is a constant. From the standard special values of the dilogarithm $\mathrm{Li}_2(1)=\pi^2/6$ and $\mathrm{Li}_2(-1)=-\pi^2/12$ and the basic limit $\lim_{z\rightarrow1}\ln(z)\tanh^{-1}(z)=0$ we obtain that $L(1)=R(1)=\pi^2/4$, so that $F(1)=0$, from which we obtain that $F(z)=0$ for $\mathrm{Re}\,z>0$, which verifies that identity (\ref{jx}) does indeed hold for $\mathrm{Re}\,z>0$.

\vspace{3mm}

\noindent 3. Applying the reduction formulas (\ref{flog}), (\ref{red33}) and (\ref{red44}), respectively, to the integral formula (\ref{forab2}) yields the following elementary evaluations:
\begin{align*}
\int_0^\infty xI_{1/2}(a x)K_{1/2}(a x) K_{1/2}^2(b x) \,\mathrm{d}x&=\frac{\pi}{4ab}\ln\bigg(1+\frac{a}{b}\bigg), \\
\int_0^\infty xI_{1/4}(a x)K_{1/4}(a x) K_{1/4}^2(b x) \,\mathrm{d}x&=\frac{\pi\sqrt{2}}{4ab}\bigg\{\tan^{-1}\bigg(\sqrt{\frac{a}{b}}\bigg)+\tanh^{-1}\bigg(\frac{a}{b}\bigg)-\tanh^{-1}\bigg(\sqrt{\frac{a}{b}}\bigg)\bigg\}\\
&=\frac{\pi\sqrt{2}}{4ab}\bigg\{\tan^{-1}\bigg(\sqrt{\frac{a}{b}}\bigg)-\coth^{-1}\bigg(\sqrt{\frac{a}{b}}+1+\sqrt{\frac{b}{a}}\bigg)\bigg\},  \\
\int_0^\infty xI_{-1/4}(a x)K_{1/4}(a x) K_{1/4}^2(b x) \,\mathrm{d}x&=\frac{\pi\sqrt{2}}{4ab}\bigg\{\tan^{-1}\bigg(\sqrt{\frac{a}{b}}\bigg)-\tanh^{-1}\bigg(\frac{a}{b}\bigg)+
\tanh^{-1}\bigg(\sqrt{\frac{a}{b}}\bigg)\bigg\}\\
&=\frac{\pi\sqrt{2}}{4ab}\bigg\{\tan^{-1}\bigg(\sqrt{\frac{a}{b}}\bigg)+\coth^{-1}\bigg(\sqrt{\frac{a}{b}}+1+\sqrt{\frac{b}{a}}\bigg)\bigg\}.
\end{align*}
In obtaining the last integral formula we used the integral formula (\ref{forab2}) to calculate the integral $\int_0^\infty xI_{-1/4}(a x)K_{-1/4}(a x) K_{-1/4}^2(b x) \,\mathrm{d}x$ and then used that $K_{-1/4}(x)=K_{1/4}(x)$ (see equation (\ref{par})) to deduce a formula for the integral $\int_0^\infty xI_{-1/4}(a x)K_{1/4}(a x) K_{1/4}^2(b x) \,\mathrm{d}x$.

\vspace{3mm}

\noindent 4. By similar considerations to part 2 of this remark, we may use the special values of the digamma function given in equation (\ref{psi2}) to obtain the following elementary evaluations from the integral formula (\ref{foraa2}):
\begin{align}
\int_0^\infty xI_{1/2}(ax)K_{1/2}^3(ax)\,\mathrm{d}x&=\frac{\pi}{4a^2}\ln(2), \nonumber \\
\int_0^\infty xI_{1/3}(ax)K_{1/3}^3(ax)\,\mathrm{d}x&=\frac{\pi}{8a^2}\big(\pi-\sqrt{3}\ln(3)\big), \nonumber 
\\
\int_0^\infty xI_{-1/3}(ax)K_{1/3}^3(ax)\,\mathrm{d}x&=\frac{\pi}{8a^2}\big(\pi+\sqrt{3}\ln(3)\big), \nonumber \\
\int_0^\infty xI_{1/4}(ax)K_{1/4}^3(ax)\,\mathrm{d}x&=\frac{\pi\sqrt{2}}{16a^2}\big(\pi-2\ln(2)\big), \nonumber \\
\int_0^\infty xI_{-1/4}(ax)K_{1/4}^3(ax)\,\mathrm{d}x&=\frac{\pi\sqrt{2}}{16a^2}\big(\pi+2\ln(2)\big). \nonumber
\end{align}    
\end{remark}

\begin{theorem}\label{cor2.6}
1. Suppose that $\mathrm{Re}\,b\geq\mathrm{Re}\,a>0$ and $\mathrm{Re}\,\alpha>-1/2$. Then
\begin{align}
\int_0^\infty I_\alpha^2(a x) K_\alpha^2(b x) \,\mathrm{d}x&=\frac{\pi}{4b}\bigg(\frac{a}{2b}\bigg)^{2\alpha}\frac{\Gamma(\alpha+\frac{1}{2})\Gamma(2\alpha+\frac{1}{2})}{\Gamma^3(\alpha+1)}\nonumber\\
&\quad\times{}_4F_3\bigg({\alpha+\frac{1}{2},2\alpha+\frac{1}{2},\alpha+\frac{1}{2},\frac{1}{2} \atop \alpha+1,2\alpha+1,\alpha+1}\;\bigg|\;\frac{a^2}{b^2}\bigg). \label{haw}
\end{align}  
2. Suppose that $\mathrm{Re}\,b>\mathrm{Re}\,a>0$ and $\mathrm{Re}\,\alpha>-1$. Then
\begin{align}
\int_0^\infty x I_\alpha^2(a x) K_\alpha^2(b x) \,\mathrm{d}x
=\frac{(a/b)^{2\alpha}}{2b^2(1+2\alpha)}\,{}_2F_1\bigg({\alpha+\frac{1}{2},1 \atop \alpha+\frac{3}{2}}\;\bigg| \;\frac{a^2}{b^2}\bigg). \label{2f1ab}
\end{align} 
\end{theorem}

\begin{remark} 1. Suppose that $\mathrm{Re}\,b\geq\mathrm{Re}\,a>0$. Then setting $\alpha=0$ in the the integral formula (\ref{haw}) yields
\begin{align}
\int_0^\infty I_0^2(ax)K_0^2(bx)\,\mathrm{d}x&=\frac{\pi^2}{4ab}\,{}_4F_3\bigg({\frac{1}{2},\frac{1}{2},\frac{1}{2},\frac{1}{2} \atop 1,1,1} \; \bigg| \;\frac{a^2}{b^2}\bigg) \label{ox1} \\
&=\sum_{k=0}^\infty\bigg(\frac{(1/2)_k}{k!}\bigg)^4\bigg(\frac{a}{b}\bigg)^{2k} \nonumber \\
&=\frac{\pi^2}{4ab}\sum_{k=0}^\infty\binom{2k}{k}^4\bigg(\frac{a}{16b}\bigg)^{2k}, \label{ox2}  
\end{align}
where we used the basic relation $(1/2)_k=(2k)!/(4^k k!)$. The formulas (\ref{ox1}) and (\ref{ox2}) generalise formulas given by \cite{b08} for the case $a=b$.

\vspace{3mm}

\noindent 2. Suppose that $\mathrm{Re}\,b>\mathrm{Re}\,a>0$ and let $n=0,1,2,\ldots$. Then, applying the reduction formula (\ref{2f1redtanh}) to the integral formula (\ref{2f1ab}) yields the elementary evaluation
\begin{align*}
\int_0^\infty xI_n^2(ax)K_n^2(bx)\,\mathrm{d}x=\frac{1}{2b^2}\bigg\{\frac{b}{a}\tanh^{-1}\bigg(\frac{a}{b}\bigg)-\sum_{k=0}^{n-1}\frac{(a/b)^{2k}}{2k+1}\bigg\},    
\end{align*} 
whilst applying the reduction formula (\ref{flog}) yields
\begin{align*}
\int_0^\infty xI_{n+1/2}^2(ax)K_{n+1/2}^2(bx)\,\mathrm{d}x=\frac{1}{4ab}\bigg\{-\ln\bigg(1-\frac{a^2}{b^2}\bigg)-\sum_{k=1}^{n}\frac{(a/b)^{2k}}{k}\bigg\}.    
\end{align*}
Here we use the convention that the empty sum is set to 0.
\end{remark}

\section{Integrals of products of four Airy functions}\label{sec3}

The Airy functions $\mathrm{Ai}(z)$ and $\mathrm{Bi}(z)$ can be expressed in terms of the modified Bessel functions of the first and second kind: for $z\in\mathbb{C}$,
\begin{align}
\mathrm{Ai}(z)&=\frac{1}{\pi}\sqrt{\frac{z}{3}}K_{1/3}\bigg(\frac{2}{3}z^{3/2}\bigg),  \label{airyaidef} \\
\mathrm{Bi}(z)&=\sqrt{\frac{z}{3}}\bigg[I_{1/3}\bigg(\frac{2}{3}z^{3/2}\bigg)+I_{-1/3}\bigg(\frac{2}{3}z^{3/2}\bigg)\bigg]   \label{airybidef}  \end{align}
(see equations 9.6.1 and 9.6.3 of \cite{olver}). As such, we can apply the integral formulas of Section \ref{sec2} to obtain integral formulas for quartic products of Airy functions. In the following theorem, we present several such formulas, all of which we believe to be new.

\begin{theorem}\label{corai}
\noindent 1. For $\mathrm{Re}\,a,\mathrm{Re}\,b>0$,
\begin{align}
\int_0^\infty \mathrm{Ai}^2(ax)\mathrm{Ai}^2(bx)\,\mathrm{d}x&=\frac{1}{12\pi^2\sqrt{ab}}\tanh^{-1}\bigg(\frac{\sqrt{ab}}{a+b}\bigg). \label{ai22}  
\end{align}   
\noindent 2. For $\mathrm{Re}\,a,\mathrm{Re}\,b>0$,
\begin{align}\int_0^\infty \mathrm{Bi}(ax)\mathrm{Ai}(ax)\mathrm{Ai}^2(bx)\,\mathrm{d}x=\frac{\pi\sqrt{3}}{10ab^2}\bigg\{5\,{}_2F_1\bigg({\frac{1}{6},1 \atop \frac{7}{6}}\;\bigg| \;\frac{a^3}{b^3}\bigg)-\frac{a^2}{b^2}\,{}_2F_1\bigg({\frac{5}{6},1 \atop \frac{11}{6}}\;\bigg| \;\frac{a^3}{b^3}\bigg)\bigg\}. \label{s57}
\end{align}
For $|a/b|<1$ the integral takes an elementary form:
\begin{align}\label{lion}
\int_0^\infty \mathrm{Bi}(ax)\mathrm{Ai}(ax)\mathrm{Ai}^2(bx)\,\mathrm{d}x=\frac{1}{12\pi^2\sqrt{ab}}\tan^{-1}\bigg(\frac{\sqrt{3ab}}{b-a}\bigg).    
\end{align}
\noindent 3. For $\mathrm{Re}\,b>\mathrm{Re}\,a>0$,
\begin{align}
\int_0^\infty \mathrm{Bi}^2(ax)\mathrm{Ai}^2(bx)\,\mathrm{d}x
&=\frac{1}{12\pi^2\sqrt{ab}}\bigg\{\tanh^{-1}\bigg(\frac{a^{3/2}}{b^{3/2}}\bigg)+3\tanh^{-1}\bigg(\sqrt{\frac{a}{b}}\bigg)\bigg\}. \label{lion2}
\end{align}
\noindent 4. For $\mathrm{Re}\,b>3\mathrm{Re}\,a>0$,
\begin{align} \label{lion3}
\int_0^\infty \mathrm{Bi}^3(ax)\mathrm{Ai}(bx)\,\mathrm{d}x&=I(\tfrac{1}{3},\tfrac{1}{3},\tfrac{1}{3})+3I(\tfrac{1}{3},\tfrac{1}{3},-\tfrac{1}{3})+3I(\tfrac{1}{3},-\tfrac{1}{3},-\tfrac{1}{3})+I(-\tfrac{1}{3},-\tfrac{1}{3},-\tfrac{1}{3}),   
\end{align}
where
\begin{align*}
I(p,q,r)&=\frac{1}{3\pi}\bigg(\frac{a}{b}\bigg)^{2(p+q+r)+1}\times\\
&\quad\times F_C^{(3)}\bigg(\frac{p+q+r}{2}+\frac{2}{3},\frac{p+q+r}{2}+\frac{1}{3};p+1,q+1,r+1;\frac{a^2}{b^2},\frac{a^2}{b^2},\frac{a^2}{b^2}\bigg). \end{align*}
\end{theorem}

\begin{remark} Setting $b=1$ and then taking the limit $a\rightarrow1^{-}$ in the integral formulas (\ref{ai22}) and (\ref{lion}) (using the basic results that $\tanh^{-1}(1/2)=\ln(3)/2$ and $\lim_{x\rightarrow\infty}\tan^{-1}(x)=\pi/2$) yields the simple evaluations
\begin{align*}
\int_0^\infty \mathrm{Ai}^4(x)\,\mathrm{d}x&=\frac{\ln(3)}{24\pi^2},    \\
\int_0^\infty \mathrm{Ai}^3(x)\mathrm{Bi}(x)\,\mathrm{d}x&=\frac{1}{24\pi}.   
\end{align*} 
These integral formulas were previously obtained by \cite{l93} and \cite{r97b}, respectively. 
\end{remark}

In our derivation of the integral formula (\ref{lion}), we will require the following elementary evaluation for a difference of two hypergeometric functions.

\begin{lemma}\label{flfl} For $|z|<1$,
\begin{align}
D(z):=5\,{}_2F_1\bigg({\frac{1}{6},1 \atop \frac{7}{6}}\;\bigg| \;z\bigg)-z^{2/3}\,{}_2F_1\bigg({\frac{5}{6},1 \atop \frac{11}{6}}\;\bigg| \;z\bigg)=\frac{5}{\sqrt{3}z^{1/6}}\tan^{-1}\bigg(\frac{\sqrt{3}z^{1/6}}{1-z^{1/3}}\bigg).    \label{dddfor}
\end{align}    
\end{lemma}

\begin{remark}
The condition $|z|<1$ is not just a technical condition required to push through the proof. Numerical tests reveal that 
equality (\ref{dddfor}) does not hold in general for $|z|>1$.    
\end{remark}

\begin{proof}
We begin by noting that, for $|z|<1$ and $\mathrm{Re}\,b>0$,
\begin{align*}{}_2F_1\bigg({1,b \atop b+1}\;\bigg| \;z\bigg)=\sum_{n=0}^\infty\frac{(1)_n(b)_n}{(b+1)_n}\frac{z^n}{n!}=\sum_{n=0}^\infty\frac{bz^n}{b+n}=\frac{b}{z^b}\int_0^z\frac{x^{b-1}}{1-x}\,\mathrm{d}x,
\end{align*}
from which we obtain that
\begin{align*}
D(z)=\frac{5}{6z^{1/6}}\int_0^z\frac{x^{-5/6}-x^{-1/6}}{1-x}\,\mathrm{d}x=\frac{5}{z^{1/6}}\int_0^{z^{1/6}}\frac{1-t^4}{1-t^6}\,\mathrm{d}t,   
\end{align*}
where we made the substitution $x=t^6$. Using the factorisations $1-t^4=(1-t^2)(1+t^2)$ and $(1-t^6)=(1-t^2)(1+t^2+t^4)$ followed by a partial fraction decomposition and a standard integration we obtain that
\begin{align*}
D(z)&=\frac{5}{z^{1/6}}\int_0^{z^{1/6}}\frac{t^2+1}{t^4+t^2+1}\,\mathrm{d}t=\frac{5}{2z^{1/6}}\int_0^{z^{1/6}}\bigg(\frac{1}{t^2-t+1}+\frac{1}{t^2+t+1}\bigg)\,\mathrm{d}t \\
&=\frac{5}{2z^{1/6}}\int_0^{z^{1/6}}\bigg(\frac{1}{(t-1/2)^2+3/4}+\frac{1}{(t+1/2)^2+3/4}\bigg)\,\mathrm{d}t\\
&=\frac{5}{2z^{1/6}}\bigg[\frac{2}{\sqrt{3}}\bigg\{\tan^{-1}\bigg(\frac{2t-1}{\sqrt{3}}\bigg)+\tan^{-1}\bigg(\frac{2t+1}{\sqrt{3}}\bigg)\bigg\}\bigg]_0^{z^{1/6}}\\
&=\frac{5}{\sqrt{3}z^{1/6}}\bigg\{\tan^{-1}\bigg(\frac{2z^{1/6}-1}{\sqrt{3}}\bigg)+\tan^{-1}\bigg(\frac{2z^{1/6}+1}{\sqrt{3}}\bigg)\bigg\}.
\end{align*}
By applying the addition formula $\tan^{-1}(u)+\tan^{-1}(v)=\tan^{-1}((u+ v)/(1- uv))$ (see \cite[equation 4.24.15]{olver}) to the above expression we obtain the desired formula (\ref{dddfor}).
\end{proof}

\begin{proof}[Proof of Theorem \ref{corai}] 1. Expressing the Airy function $\mathrm{Ai}(x)$ in terms of the modified Bessel function of the second kind using the relation (\ref{airyaidef}) gives that
\begin{align}
\int_0^\infty \mathrm{Ai}^2(ax)\mathrm{Ai}^2(bx)\,\mathrm{d}x&=\frac{ab}{9\pi^4}\int_0^\infty x^{2}K_{1/3}^2\bigg(\frac{2}{3}(ax)^{3/2}\bigg)K_{1/3}^2\bigg(\frac{2}{3}(bx)^{3/2}\bigg)\,\mathrm{d}x \nonumber\\
&=\frac{ab}{6\pi^4}\int_0^\infty yK_{1/3}^2(a^{3/2}y)K_{1/3}^2(b^{3/2}y)\,\mathrm{d}y, \label{subb0}
\end{align}
where we made the substitution $y=2x^{3/2}/3$. We now evaluate the integral (\ref{subb0}) using the formula (\ref{kab13}), which yields the integral formula (\ref{ai22}).

\vspace{3mm}

\noindent 2. We now argue similarly to in part 1 of the proof by expressing the Airy functions $\mathrm{Ai}(z)$ and $\mathrm{Bi}(z)$ in terms of the modified Bessel functions of the first and second kind using the relations (\ref{airyaidef}) and (\ref{airybidef}) to obtain that
\begin{align}
\int_0^\infty \mathrm{Bi}(ax)\mathrm{Ai}(ax)\mathrm{Ai}^2(bx)\,\mathrm{d}x&=\frac{ab}{6\pi^3}\bigg\{\int_0^\infty yK_{1/3}^2(b^{3/2}y)K_{1/3}(a^{3/2}y)I_{1/3}(a^{3/2}y)\,\mathrm{d}y \nonumber\\
&\quad+\int_0^\infty yK_{1/3}^2(b^{3/2}y)K_{1/3}(a^{3/2}y)I_{-1/3}(a^{3/2}y)\,\mathrm{d}y\bigg\} \nonumber \\
&=\frac{ab}{6\pi^3}\bigg\{\int_0^\infty yK_{1/3}^2(b^{3/2}y)K_{1/3}(a^{3/2}y)I_{1/3}(a^{3/2}y)\,\mathrm{d}y \nonumber\\
&\quad+\int_0^\infty yK_{-1/3}^2(b^{3/2}y)K_{-1/3}(a^{3/2}y)I_{-1/3}(a^{3/2}y)\,\mathrm{d}y\bigg\},\label{subb2}
\end{align}
where in the second step we used that $K_{1/3}(z)=K_{-1/3}(z)$ (see equation (\ref{par})). Evaluating the integrals (\ref{subb2}) using the integral formula (\ref{forab2}) and then simplifying now gives that
\begin{align}
\int_0^\infty \mathrm{Bi}(ax)\mathrm{Ai}(ax)\mathrm{Ai}^2(bx)\,\mathrm{d}x=
\frac{\pi\sqrt{3}}{10ab^2}\bigg\{5\,{}_2F_1\bigg({\frac{1}{6},1 \atop \frac{7}{6}}\;\bigg| \;\frac{a^3}{b^3}\bigg)-\frac{a^2}{b^2}\,{}_2F_1\bigg({\frac{5}{6},1 \atop \frac{11}{6}}\;\bigg| \;\frac{a^3}{b^3}\bigg)\bigg\}, \label{bomdia}
\end{align}
which completes the derivation of the integral formula (\ref{s57}). When $|a/b|<1$ we can simplify the expression (\ref{bomdia}) using formula (\ref{dddfor}) of Lemma \ref{flfl} to obtain the integral formula (\ref{lion}).

\vspace{3mm}

\noindent 3. From the relation (\ref{airybidef}) we have that
\begin{align}
\int_0^\infty \mathrm{Bi}^2(ax)\mathrm{Ai}^2(bx)\,\mathrm{d}x&=\frac{ab}{6\pi^2}\bigg\{\int_0^\infty yI_{1/3}^2(a^{3/2}y)K_{1/3}^2(b^{3/2}y)\,\mathrm{d}y\nonumber\\
&\quad+2\int_0^\infty yI_{1/3}(a^{3/2}y)I_{-1/3}(a^{3/2}y)K_{1/3}^2(b^{3/2}y)\,\mathrm{d}y\nonumber\\
&\quad+\int_0^\infty yI_{-1/3}^2(a^{3/2}y)K_{1/3}^2(b^{3/2}y)\,\mathrm{d}y\bigg\}. \label{il30}  
\end{align}
Using relation (\ref{ikrel}) we have that $I_{-1/3}(z)=I_{1/3}(z)+(\sqrt{3}/\pi)K_{1/3}(z)$, and applying this formula to equation (\ref{il30}) gives that
\begin{align}
\int_0^\infty \mathrm{Bi}^2(ax)\mathrm{Ai}^2(bx)\,\mathrm{d}x&=\frac{ab}{6\pi^2}\bigg\{3\int_0^\infty yI_{1/3}^2(a^{3/2}y)K_{1/3}^2(b^{3/2}y)\,\mathrm{d}y\nonumber\\
&\quad+\int_0^\infty yI_{-1/3}^2(a^{3/2}y)K_{1/3}^2(b^{3/2}y)\,\mathrm{d}y\nonumber\\
&\quad+\frac{2\sqrt{3}}{\pi}\int_0^\infty yI_{1/3}(a^{3/2}y)K_{1/3}(a^{3/2}y)K_{1/3}^2(b^{3/2}y)\,\mathrm{d}y\bigg\} \nonumber\\
&=\frac{ab}{6\pi^2}\bigg\{4\int_0^\infty yI_{1/3}^2(a^{3/2}y)K_{1/3}^2(b^{3/2}y)\,\mathrm{d}y\nonumber\\
&\quad+\frac{4\sqrt{3}}{\pi}\int_0^\infty yI_{1/3}(a^{3/2}y)K_{1/3}(a^{3/2}y)K_{1/3}^2(b^{3/2}y)\,\mathrm{d}y\nonumber\\
&\quad+\frac{3}{\pi^2}\int_0^\infty yK_{1/3}^2(a^{3/2}y)K_{1/3}^2(b^{3/2}y)\,\mathrm{d}y\bigg\}.\label{dttd}
\end{align}
Evaluating the integrals in (\ref{dttd}) using the integral formulas (\ref{2f1ab}), (\ref{forab2}) and (\ref{kab13}) gives that
\begin{align}
\int_0^\infty \mathrm{Bi}^2(ax)\mathrm{Ai}^2(bx)\,\mathrm{d}x&=\frac{ab}{6\pi^2}\bigg\{\frac{6(a/b)}{5b^3}\,{}_2F_1\bigg({1,\frac{5}{6} \atop \frac{11}{6}}\;\bigg| \;\frac{a^3}{b^3}\bigg)\nonumber\\
&\quad+\frac{4\sqrt{3}}{\pi}\cdot \frac{\pi}{2\sqrt{3}b^3}\bigg[\frac{b^{3/2}}{a^{3/2}}\tanh^{-1}\bigg(\frac{a^{3/2}}{b^{3/2}}\bigg)-\frac{3(a/b)}{5}\,{}_2F_1\bigg({1,\frac{5}{6} \atop \frac{11}{6}}\;\bigg| \;\frac{a^3}{b^3}\bigg)\bigg]\nonumber\\
&\quad+\frac{3}{\pi^2}\cdot \frac{\pi^2}{2(ab)^{3/2}}\tanh^{-1}\bigg(\frac{\sqrt{ab}}{a+b}\bigg)\bigg\} \nonumber \\
&=\frac{1}{12\pi^2\sqrt{ab}}\bigg\{4\tanh^{-1}\bigg(\frac{a^{3/2}}{b^{3/2}}\bigg)+3\tanh^{-1}\bigg(\frac{\sqrt{ab}}{a+b}\bigg)\bigg\}, \label{addd}
\end{align}
and simplifying the expression in (\ref{addd}) using the addition formula (\ref{addtanh}) yields formula (\ref{lion2}).

\vspace{3mm}

\noindent 4. Arguing as before, we have that
\begin{align*}
\int_0^\infty \mathrm{Bi}^3(ax)\mathrm{Ai}(bx)\,\mathrm{d}x&=\frac{ab}{6\pi}\bigg\{\int_0^\infty yI_{1/3}^3(a^{3/2}y)K_{1/3}(b^{3/2}y)\,\mathrm{d}y\\
&\quad+3\int_0^\infty yI_{1/3}^2(a^{3/2}y)I_{-1/3}(a^{3/2}y)K_{1/3}(b^{3/2}y)\,\mathrm{d}y\\
&\quad+3\int_0^\infty yI_{1/3}(a^{3/2}y)I_{-1/3}^2(a^{3/2}y)K_{1/3}(b^{3/2}y)\,\mathrm{d}y\\
&\quad+\int_0^\infty yI_{-1/3}^3(a^{3/2}y)K_{1/3}(b^{3/2}y)\,\mathrm{d}y\bigg\},
\end{align*}
and evaluating the integrals using the integral formula (\ref{lion5}) yields formula (\ref{lion3}).
\end{proof}

\section{Proofs from Section \ref{sec2}}\label{sec4}

\begin{proof}[Proof of Theorem \ref{thm1.1}] 1. We begin by noting that, by an application of the integral formula 6.576(4) of \cite{g07}, the Mellin transform of the function $f_{\mu,\nu,\alpha}(x)=K_\mu(a x)K_\nu(a x)$ is given by
\begin{align}
\mathcal{M}\{f_{\mu,\nu,a}\}(s)&=\int_0^\infty x^{s-1} K_\mu(a x)K_\nu(a x)\,\mathrm{d}x \nonumber \\
&=\frac{2^{s-3}}{a^s \Gamma(s)}\Gamma\bigg(\frac{s+\mu+\nu}{2}\bigg)\Gamma\bigg(\frac{s-\mu+\nu}{2}\bigg)\Gamma\bigg(\frac{s+\mu-\nu}{2}\bigg)\Gamma\bigg(\frac{s-\mu-\nu}{2}\bigg), \label{mellin1}
\end{align}
which exists for $\mathrm{Re}\,a>0$ and $\mathrm{Re}\,s>|\mathrm{Re}\, \mu|+|\mathrm{Re}\, \nu|$.

By the standard formula for the Mellin transform of the product of two functions we then have that, for $\mathrm{Re}\,a,\mathrm{Re}\,b>0$ and $\mathrm{Re}\, s>|\mathrm{Re}\,\alpha|+|\mathrm{Re}\,\beta|+|\mathrm{Re}\,\gamma|+|\mathrm{Re}\,\delta|$,
\begin{align*}
I&=\int_0^\infty x^{s-1} K_\alpha(a x) K_\beta(a x) K_\gamma(b x) K_\delta(b x) \,\mathrm{d}x\\
&=\frac{1}{2\pi\mathrm{i}}\int_{c-\mathrm{i}\infty}^{c+\mathrm{i}\infty}\mathcal{M}\{f_{\alpha,\beta,a}\}(s-r)\mathcal{M}\{f_{\gamma,\delta,b}\}(r)\,\mathrm{d}r\\
&=\frac{1}{2\pi\mathrm{i}}\int_{c-\mathrm{i}\infty}^{c+\mathrm{i}\infty} \frac{2^{s-6}}{a^s}\bigg(\frac{a}{b}\bigg)^r\frac{1}{\Gamma(s-r)\Gamma(r)}\\
&\quad\times\Gamma\bigg(\frac{s-r+\alpha+\beta}{2}\bigg)\Gamma\bigg(\frac{s-r-\alpha+\beta}{2}\bigg)\Gamma\bigg(\frac{s-r+\alpha-\beta}{2}\bigg)\Gamma\bigg(\frac{s-r-\alpha-\beta}{2}\bigg)\\
&\quad\times \Gamma\bigg(\frac{r+\gamma+\delta}{2}\bigg)\Gamma\bigg(\frac{r-\gamma+\delta}{2}\bigg)\Gamma\bigg(\frac{r+\gamma-\delta}{2}\bigg)\Gamma\bigg(\frac{r-\gamma-\delta}{2}\bigg)\,\mathrm{d}r,
\end{align*}
where $c>|\mathrm{Re}\,\gamma|+|\mathrm{Re}\,\delta|$. By applying the duplication formula 
\begin{align}\label{dupfor} \Gamma(2z)=\pi^{-1/2}2^{2z-1}\Gamma(z)\Gamma(z+1/2),\quad 2z\notin\mathbb{Z}
\end{align}
(see \cite[Section 5.5(iii)]{olver}) we now get that
\begin{align}
I&=\frac{\pi}{16a^s}\cdot \frac{1}{2\pi\mathrm{i}}\int_{c-\mathrm{i}\infty}^{c+\mathrm{i}\infty}\bigg(\frac{a}{b}\bigg)^r\frac{1}{\Gamma((s-r)/2)\Gamma((s-r+1)/2)\Gamma(r/s)\Gamma((r+1)/2)}\nonumber\\
&\quad\times\Gamma\bigg(\frac{s-r+\alpha+\beta}{2}\bigg)\Gamma\bigg(\frac{s-r-\alpha+\beta}{2}\bigg)\Gamma\bigg(\frac{s-r+\alpha-\beta}{2}\bigg)\Gamma\bigg(\frac{s-r-\alpha-\beta}{2}\bigg)\nonumber\\
&\quad\times \Gamma\bigg(\frac{r+\gamma+\delta}{2}\bigg)\Gamma\bigg(\frac{r-\gamma+\delta}{2}\bigg)\Gamma\bigg(\frac{r+\gamma-\delta}{2}\bigg)\Gamma\bigg(\frac{r-\gamma-\delta}{2}\bigg)\,\mathrm{d}r\nonumber\\
&=\frac{\pi}{8a^s}\cdot \frac{1}{2\pi\mathrm{i}}\int_{c-\mathrm{i}\infty}^{c+\mathrm{i}\infty}\bigg(\frac{a^2}{b^2}\bigg)^t\frac{1}{\Gamma(s/2-t)\Gamma((s+1)/2-t)\Gamma(t)\Gamma(t+1/2)}\nonumber\\
&\quad\times\Gamma\bigg(\frac{s+\alpha+\beta}{2}-t\bigg)\Gamma\bigg(\frac{s-\alpha+\beta}{2}-t\bigg)\Gamma\bigg(\frac{s+\alpha-\beta}{2}-t\bigg)\Gamma\bigg(\frac{s-\alpha-\beta}{2}-t\bigg)\nonumber\\
&\quad\times \Gamma\bigg(t+\frac{\gamma+\delta}{2}\bigg)\Gamma\bigg(t+\frac{-\gamma+\delta}{2}\bigg)\Gamma\bigg(t+\frac{\gamma-\delta}{2}\bigg)\Gamma\bigg(t+\frac{-\gamma-\delta}{2}\bigg)\,\mathrm{d}t \nonumber\\
&=\frac{\pi}{8a^s}G_{6,6}^{4,4}\bigg(\frac{a^2}{b^2}\;\bigg|\;{\frac{2-\gamma-\delta}{2},\frac{2+\gamma-\delta}{2},\frac{2+\gamma-\delta}{2},\frac{2+\gamma+\delta}{2},\frac{s}{2},\frac{s+1}{2} \atop \frac{s+\alpha+\beta}{2},\frac{s-\alpha+\beta}{2},\frac{s+\alpha-\beta}{2},\frac{s-\alpha-\beta}{2},1,\frac{1}{2}}\bigg), \nonumber
\end{align}
where in obtaining the second equality we made the change of variables $r=2t$ and in obtaining the final equality we used the contour integral representation (\ref{mdef}) of the Meijer $G$-function. We have thus derived formula (\ref{for1}).

\vspace{3mm}

\noindent 2. We begin by noting that, by an application of the integral formula 6.576(5) of \cite{g07}, the Mellin transform of the function $g_{\mu,\nu,a,d}(x)=K_\mu(a x)I_\nu(d x)$ is given by
\begin{align*}
\mathcal{M}\{g_{\mu,\nu,a,d}\}(s)&=\int_0^\infty x^{s-1} K_\mu(a x)I_\nu(d x)\,\mathrm{d}x\\
&=\frac{2^{s}d^\nu}{4a^{s+\nu}}\frac{\Gamma\big(\frac{s+\mu+\nu}{2}\big)\Gamma\big(\frac{s-\mu+\nu}{2}\big)}{\Gamma(\nu+1)}\,{}_2F_1\bigg(\frac{s+\mu+\nu}{2},\frac{s-\mu+\nu}{2};\nu+1;\frac{d^2}{a^2}\bigg),
\end{align*}
which exists for $\mathrm{Re}\,a>\mathrm{Re}\,d>0$ and $\mathrm{Re}\,s>\mathrm{Re}(|\mu|-\nu)$. Arguing as in part (i) we then have that
\begin{align}
I(d,a,b,b)&=\int_0^\infty x^{s-1} I_\alpha(d x) K_\beta(a x) K_\gamma(b x) K_\delta(b x) \,\mathrm{d}x\nonumber\\
&=\frac{1}{2\pi\mathrm{i}}\int_{c-\mathrm{i}\infty}^{c+\mathrm{i}\infty}\mathcal{M}\{g_{\beta,\alpha,a,d}\}(s-r)\mathcal{M}\{f_{\gamma,\delta,b}\}(r)\,\mathrm{d}r\nonumber\\
&=\frac{1}{2\pi\mathrm{i}}\int_{c-\mathrm{i}\infty}^{c+\mathrm{i}\infty}\frac{2^{s-5}}{a^s}\bigg(\frac{d}{a}\bigg)^\alpha\bigg(\frac{a}{b}\bigg)^r\frac{\Gamma\big(\frac{s-r+\alpha+\beta}{2}\big)\Gamma\big(\frac{s-r+\alpha-\beta}{2}\big)}{\Gamma(r)\Gamma(\alpha+1)}\nonumber\\
&\quad\times\Gamma\bigg(\frac{r+\gamma+\delta}{2}\bigg)\Gamma\bigg(\frac{r-\gamma+\delta}{2}\bigg)\Gamma\bigg(\frac{r+\gamma-\delta}{2}\bigg)\Gamma\bigg(\frac{r-\gamma-\delta}{2}\bigg)\nonumber\\
&\quad\times {}_2F_1\bigg(\frac{s-r+\alpha+\beta}{2},\frac{s-r+\alpha-\beta}{2};\alpha+1;\frac{d^2}{a^2}\bigg)\,\mathrm{d}r, \label{limad}
\end{align}
where $c>|\mathrm{Re}\,\beta|-\mathrm{Re}\,\alpha$. On letting $d\rightarrow a^{-}$ in (\ref{limad}) we obtain that
\begin{align}
I(a,a,b,b)&=\frac{1}{2\pi\mathrm{i}}\int_{c-\mathrm{i}\infty}^{c+\mathrm{i}\infty}\frac{2^{s-5}}{a^s}\bigg(\frac{a}{b}\bigg)^r\frac{\Gamma\big(\frac{s-r+\alpha+\beta}{2}\big)\Gamma\big(\frac{s-r+\alpha-\beta}{2}\big)}{\Gamma(r)\Gamma(\alpha+1)}\nonumber\\
&\quad\times\Gamma\bigg(\frac{r+\gamma+\delta}{2}\bigg)\Gamma\bigg(\frac{r-\gamma+\delta}{2}\bigg)\Gamma\bigg(\frac{r+\gamma-\delta}{2}\bigg)\Gamma\bigg(\frac{r-\gamma-\delta}{2}\bigg)\nonumber\\
&\quad\times {}_2F_1\bigg(\frac{s-r+\alpha+\beta}{2},\frac{s-r+\alpha-\beta}{2};\alpha+1;1\bigg)\,\mathrm{d}r \nonumber \\
&=\frac{1}{2\pi\mathrm{i}}\int_{c-\mathrm{i}\infty}^{c+\mathrm{i}\infty}\frac{2^{s-5}}{a^s}\bigg(\frac{a}{b}\bigg)^r\frac{\Gamma(r-s+1)}{\Gamma(r)}\frac{\Gamma\big(\frac{s-r+\alpha+\beta}{2}\big)\Gamma\big(\frac{s-r+\alpha-\beta}{2}\big)}{\Gamma\big(\frac{2-s+r+\alpha+\beta}{2}\big)\Gamma\big(\frac{2-s+r+\alpha-\beta}{2}\big)}\nonumber\\
&\quad\times\Gamma\bigg(\frac{r+\gamma+\delta}{2}\bigg)\Gamma\bigg(\frac{r-\gamma+\delta}{2}\bigg)\Gamma\bigg(\frac{r+\gamma-\delta}{2}\bigg)\Gamma\bigg(\frac{r-\gamma-\delta}{2}\bigg)\,\mathrm{d}r \nonumber \\
&=\frac{1}{8a^s}\cdot\frac{1}{2\pi\mathrm{i}}\int_{c-\mathrm{i}\infty}^{c+\mathrm{i}\infty}\bigg(\frac{a^2}{b^2}\bigg)^t\frac{\Gamma\big(t+\frac{1-s}{2}\big)\Gamma\big(t+\frac{2-s}{2}\big)\Gamma\big(\frac{s+\alpha+\beta}{2}-t\big)\Gamma\big(\frac{s+\alpha-\beta}{2}-t\big)}{\Gamma(t)\Gamma\big(t+\frac{1}{2}\big)\Gamma\big(\frac{2-s+\alpha+\beta}{2}+t\big)\Gamma\big(\frac{2-s+\alpha-\beta}{2}+t\big)}\nonumber\\
&\quad\times\Gamma\bigg(t+\frac{\gamma+\delta}{2}\bigg)\Gamma\bigg(t+\frac{-\gamma+\delta}{2}\bigg)\Gamma\bigg(t+\frac{\gamma-\delta}{2}\bigg)\Gamma\bigg(t-\frac{\gamma+\delta}{2}\bigg)\,\mathrm{d}t, \label{sept}
\end{align}
where we obtained the second equality using identity (\ref{2f11}) and we obtained the final equality by using the duplication formula (\ref{dupfor}) for the gamma function and the change of variables $r=2t$. Comparing the expression (\ref{sept}) and the integral representation (\ref{mdef}) of the $G$-function now yields formula (\ref{k3igen}).

\vspace{3mm}

\noindent 3. We begin by recalling the power series representation of the product of two modified Bessel functions of the first kind (\cite[equation 10.31.3]{olver}):
\begin{align*}
I_\nu(x)I_\mu(x)=\bigg(\frac{x}{2}\bigg)^{\nu+\mu}\sum_{k=0}^\infty\frac{(\nu+\mu+k+1)_k}{k!\Gamma(\nu+k+1)\Gamma(\mu+k+1)}\bigg(\frac{x}{2}\bigg)^{2k}.    
\end{align*}
On writing
\begin{equation*}
(a+k)_k=\frac{\Gamma(a+2k)}{\Gamma(a+k)}=\frac{2^{a+2k-1}\Gamma(k+\frac{a}{2})\Gamma(k+\frac{a+1}{2})}{\sqrt{\pi}\Gamma(k+a)}    
\end{equation*}
we can obtain an alternative expression for the product $I_\nu(x)I_\mu(x)$ that will be useful for our purposes:
\begin{align}
I_\nu(x)I_\mu(x)=\frac{1}{\sqrt{\pi}}\sum_{k=0}^\infty \frac{\Gamma\big(k+\frac{\nu+\mu+1}{2}\big)\Gamma\big(k+\frac{\nu+\mu+2}{2}\big)}{k!\Gamma(\nu+k+1)\Gamma(\mu+k+1)\Gamma(\nu+\mu+k+1)}  x^{2k+\mu+\nu}. \label{sept7}
\end{align}
Applying equation (\ref{sept7}) and an interchange in the order of integration and summation now yields the expression
\begin{align*}
I&=\int_0^\infty x^{s-1} I_\alpha(a x) I_\beta(a x) K_\gamma(b x) K_\delta(b x) \,\mathrm{d}x\\
&=\frac{1}{\sqrt{\pi}}\sum_{k=0}^\infty \frac{a^{2k+\lambda}\Gamma\big(k+\frac{\lambda+1}{2}\big)\Gamma\big(k+\frac{\lambda+2}{2}\big)}{k!\Gamma(\alpha+k+1)\Gamma(\beta+k+1)\Gamma(\lambda+k+1)}\int_0^\infty x^{2k+\lambda+s-1}  K_\gamma(b x) K_\delta(b x) \,\mathrm{d}x.
\end{align*}
Evaluating the integral using the Mellin transform formula (\ref{mellin1}) followed by an application of the duplication formula (\ref{dupfor}) gives that
\begin{align*}
I&=\frac{1}{8a^s\sqrt{\pi}}\sum_{k=0}^\infty \frac{\Gamma\big(k+\frac{\lambda+1}{2}\big)\Gamma\big(k+\frac{\lambda+2}{2}\big)}{k!\Gamma(k+\alpha+1)\Gamma(k+\beta+1)\Gamma(k+\lambda+1)\Gamma(2k+\lambda+s)}   \\
&\quad\times\Gamma\big(k+\tfrac{\lambda+\gamma+\delta}{2}\big)\Gamma\big(k+\tfrac{\lambda-\gamma+\delta}{2}\big)\Gamma\big(k+\tfrac{\lambda+\gamma-\delta}{2}\big)\Gamma\big(k+\tfrac{\lambda-\gamma-\delta}{2}\big)\bigg(\frac{2a}{b}\bigg)^{2k+\lambda+s}\\
&=\frac{1}{4a^s}\sum_{k=0}^\infty \frac{\Gamma\big(k+\frac{\lambda+1}{2}\big)\Gamma\big(k+\frac{\lambda+2}{2}\big)}{k!\Gamma(k+\alpha+1)\Gamma(k+\beta+1)\Gamma(k+\lambda+1)\Gamma\big(k+\frac{\lambda+s}{2}\big)\Gamma\big(k+\frac{\lambda+s+1}{2}\big)}   \\
&\quad\times\Gamma\big(k+\tfrac{\lambda+\gamma+\delta}{2}\big)\Gamma\big(k+\tfrac{\lambda-\gamma+\delta}{2}\big)\Gamma\big(k+\tfrac{\lambda+\gamma-\delta}{2}\big)\Gamma\big(k+\tfrac{\lambda-\gamma-\delta}{2}\big)\bigg(\frac{a}{b}\bigg)^{2k+\lambda+s}.
\end{align*}
By expressing the gamma functions in terms of Pochammer functions using the relation $\Gamma(u+k)=\Gamma(u)(u)_k$ and comparing the resulting expression to the series representation of the regularized generalized hypergeometric function (see (\ref{gauss}) and (\ref{gauss2})) we obtain the desired formula (\ref{iikkgen}) for the integral.

\vspace{3mm}

\noindent 4. This is a special case of formula (\ref{thm1.2for}) of Theorem \ref{thm1.2}.
\end{proof}

\begin{proof}[Proof of Theorem \ref{thm1.2}] Using the power series representation (\ref{idef}) of the modified Bessel function and an interchange in the order of integration and summation yields the expression
\begin{align*}
I&=\int_0^\infty x^{s-1}\bigg\{\prod_{k=1}^n I_{\alpha_k}(a_kx)\bigg\}K_\beta(bx)\,\mathrm{d}x\\
&=\sum_{k_1=0}^\infty\cdots\sum_{k_n=0}^\infty\frac{1}{k_1!\Gamma(\alpha_1+k_1+1)\cdots k_n!\Gamma(\alpha_n+k_n+1)}\\
&\quad\times\int_0^\infty x^{s-1}\bigg(\frac{a_1x}{2}\bigg)^{\alpha_1+2k_1}\cdots \bigg(\frac{a_nx}{2}\bigg)^{\alpha_n+2k_n} K_\beta(bx)\,\mathrm{d}x.
\end{align*}
Evaluating the integral using the integral formula
\begin{equation*}
\int_0^\infty x^{\mu-1}K_\nu(x)\,\mathrm{d}x=2^{\mu-2}\Gamma\bigg(\frac{\mu-\nu}{2}\bigg) \Gamma\bigg(\frac{\mu+\nu}{2}\bigg), \quad \mathrm{Re}\,\nu<\mathrm{Re}\,\mu   
\end{equation*}
(see \cite[equation 10.43.19]{olver}) gives that
\begin{align*}
I&=\frac{2^{s-2}}{b^s}\sum_{k_1=0}^\infty\cdots\sum_{k_n=0}^\infty \frac{\Gamma\big(k_1+\cdots+k_n+\frac{s+\alpha_1+\cdots+\alpha_n+\beta}{2}\big)\Gamma\big(k_1+\cdots+k_n+\frac{s+\alpha_1+\cdots+\alpha_n-\beta}{2}\big)}{\Gamma(\alpha_1+k_1+1)\cdots \Gamma(\alpha_n+k_n+1)(k_1!\cdots k_n!)} \\
&\quad \times \bigg(\frac{a_1}{b^2}\bigg)^{\alpha_1+k_1}\cdots \bigg(\frac{a_n}{b^2}\bigg)^{\alpha_n+k_n}.
\end{align*}
By expressing the gamma functions in terms of Pochammer functions using the relation $\Gamma(u+k)=\Gamma(u)(u)_k$ and comparing the resulting expression to the series representation (\ref{hyp}) of the Lauricella function $F_C$ yields the integral formula (\ref{thm1.2for}).
\end{proof}

\begin{proof}[Proof of Theorem \ref{cor1.3}] 1. By applying formula (\ref{fgh}) to the integral formula (\ref{corf1}) (with $\alpha=\beta$ and $s=1$) we get that
\begin{align}
&\int_0^\infty K_\alpha^2(a x) K_\alpha^2(b x) \,\mathrm{d}x=\frac{\pi}{8a}G_{4,4}^{3,3}\bigg(\frac{a^2}{b^2}\;\bigg|\;{1+\alpha,1-\alpha,1,1 \atop \frac{s}{2}+\alpha,\frac{s}{2}-\alpha,\frac{1}{2},\frac{1}{2}}\bigg)\nonumber\\
&\quad=\frac{\pi}{8a}\bigg\{\frac{\Gamma(-\alpha)\Gamma(-2\alpha)\Gamma(\frac{1}{2})\Gamma(\frac{1}{2}+\alpha)\Gamma(\frac{1}{2}+2\alpha)}{\Gamma(1+\alpha)\Gamma(\frac{1}{2}-\alpha)}\bigg(\frac{a}{2}\bigg)^{1+2\alpha} {}_4F_3\bigg({\frac{1}{2},\alpha+\frac{1}{2},\alpha+\frac{1}{2},2\alpha+\frac{1}{2} \atop 2\alpha+1,\alpha+1,\alpha+1}\;\bigg|\;\frac{a^2}{b^2}\bigg)\nonumber\\
&\quad\quad+\frac{\Gamma(\alpha)\Gamma(2\alpha)\Gamma(\frac{1}{2})\Gamma(\frac{1}{2}-\alpha)\Gamma(\frac{1}{2}-2\alpha)}{\Gamma(1-\alpha)\Gamma(\frac{1}{2}+\alpha)}\bigg(\frac{a}{b}\bigg)^{1-2\alpha}{}_4F_3\bigg({\frac{1}{2},\frac{1}{2}-2\alpha,\frac{1}{2}-\alpha,\frac{1}{2}-\alpha \atop 1-2\alpha,1-\alpha,1-\alpha}\;\bigg|\;\frac{a^2}{b^2}\bigg)\nonumber\\
&\quad\quad+\frac{\Gamma(\alpha)\Gamma(-\alpha)\Gamma(\frac{1}{2})\Gamma(\frac{1}{2}-\alpha)\Gamma(\frac{1}{2}+\alpha)}{\Gamma(1)\Gamma(\frac{1}{2})}\bigg(\frac{a}{b}\bigg){}_4F_3\bigg({\frac{1}{2},\frac{1}{2},\alpha+\frac{1}{2},\frac{1}{2}-\alpha \atop 1,1-\alpha,\alpha+1}\;\bigg|\;\frac{a^2}{b^2}\bigg)\bigg\}.\label{ratio1}
\end{align}
We obtain formula (\ref{slv1}) by simplifying the ratios of gamma functions given in (\ref{ratio1}) by using the reflection formula
\begin{align}
\Gamma(z)\Gamma(1-z)=\frac{\pi}{\sin(\pi z)}, \quad z\notin\mathbb{Z}, \label{ref0}   
\end{align}
(see \cite[equation 5.5.3]{olver}) from which we get that
\begin{align}
\Gamma(\tfrac{1}{2}+z)\Gamma(\tfrac{1}{2}-z)=\frac{\pi}{\cos(\pi z)}, \quad z+1/2\notin\mathbb{Z},   \label{refl} 
\end{align}
the duplication formula (\ref{dupfor}), the standard special values $\Gamma(1/2)=\sqrt{\pi}$, $\Gamma(1)=1$ and the identity $\sin(2\pi z)=2\sin(\pi z)\cos(\pi z)$.

\vspace{3mm}

\noindent 2. We derive the integral formula (\ref{forab}) by a similar approach to our derivation of formula (\ref{slv1}), but proceed more efficiently, omitting some details of the calculations. By applying the reduction formula (\ref{lukeformula0}) to the $G$-function in the integral formula (\ref{corf1}) (with $\alpha=\beta$ and $s=2$) we get that
\begin{align}
\int_0^\infty x K_\alpha^2(a x) K_\alpha^2(b x) \,\mathrm{d}x=\frac{\pi}{8a^2}G_{4,4}^{3,3}\bigg(\frac{a^2}{b^2}\;\bigg|\;{1+\alpha,1-\alpha,1,\frac{3}{2} \atop 1+\alpha,1-\alpha,1,\frac{1}{2}}\bigg).  \label{qc} 
\end{align}
We can apply equation (\ref{fgh}) to express the $G$-function in (\ref{qc}) as a sum of three ${}_4F_3$ functions, in a similar manner to how we proceeded in deriving the integral formula (\ref{slv1}). Here, due to repeated factors the ${}_4F_3$ functions reduce to ${}_2F_1$ functions by the reduction formula (\ref{fred}). As an example, we have that
\begin{align*}
{}_4F_3\bigg({1-\alpha,1+\alpha,1,\frac{1}{2} \atop 1-\alpha,1+\alpha,\frac{3}{2}}\;\bigg|\;\frac{a^2}{b^2}\bigg)&={}_2F_1\bigg({1,\frac{1}{2} \atop \frac{3}{2}}\;\bigg|\;\frac{a^2}{b^2}\bigg)=\frac{b}{a}\tanh^{-1}\bigg(\frac{a}{b}\bigg),   
\end{align*}
where the second equality follows from the reduction formula (\ref{2f1redtanh}). With these considerations we obtain the desired formula (\ref{forab}).

We derive the integral formula (\ref{foraa}) by taking the limit $a/b\rightarrow1^{-}$ in (\ref{forab}). To this end, we will apply the asymptotic approximation (\ref{2f1exp}) together with the basic asymptotic approximations
\begin{align}\ln(1-z^2)=\ln(1-z)+\ln(1+z)= \ln(1-z)+\ln(2)+o(1), \quad z\rightarrow 1^{-}, \label{lll0}
\end{align}
and
\begin{align}
\tanh^{-1}(z)=\frac{1}{2}\big(\ln(1+z)-\ln(1-z)\big)=\frac{1}{2}\big(\ln(2)-\ln(1-z)\big)+o(1), \quad z\rightarrow 1^{-}.   \label{lll2}  
\end{align}
Denoting the integral in equation (\ref{forab}) by $I(\alpha,z)$, where $z=a/b$, we then get that
\begin{align*}
\lim_{z\rightarrow 1^{-}} I(\alpha,z)&= \frac{\pi^2}{8a^2}\csc^2(\pi\alpha)\lim_{z\rightarrow 1^{-}}\bigg\{-\frac{1}{1+2\alpha}\frac{\Gamma(\alpha+\frac{3}{2})}{\Gamma(\alpha+\frac{1}{2})}\bigg(\ln(1-z^2)+\psi(1)+\psi\bigg(\alpha+\frac{1}{2}\bigg)+2\gamma\bigg)\\
&\quad-\frac{1}{1-2\alpha}\frac{\Gamma(\frac{3}{2}-\alpha)}{\Gamma(\frac{1}{2}-\alpha)}\bigg(\ln(1-z^2)+\psi(1)+\psi\bigg(\frac{1}{2}-\alpha\bigg)+2\gamma\bigg)-2\tanh^{-1}(z)\bigg\}\\
&=-\frac{\pi^2}{16a^2}\mathrm{csc}^2(\pi\alpha)\bigg\{\psi\bigg(\frac{1}{2}-\alpha\bigg)+\psi\bigg(\frac{1}{2}+\alpha\bigg)+4\ln(2)+2\gamma\bigg\},
\end{align*}
where we obtained the second equality by using that $\Gamma(3/2\pm\alpha)/\Gamma(1/2\pm\alpha)=1/2\pm\alpha$, the asymptotic approximations (\ref{lll0}) and (\ref{lll2}) and the standard formula $\psi(1)=-\gamma$.

We now derive the integral formula (\ref{li2}). In addition to our assumption that $\mathrm{Re}\,a,\mathrm{Re}\,b>0$ we will also assume that $a\not=b$ and $|a/b|\leq1$. We will derive the integral formula (\ref{li2}) by taking the limit $\alpha\rightarrow0$ in the integral formula (\ref{forab}). We begin by noting that 
\begin{equation}\csc^2(\pi\alpha)= \frac{1}{\pi^2\alpha^2}+O(1), \quad \alpha\rightarrow0. \label{wag1}
\end{equation}
To ease notation, we will let $z=a/b$ for the remainder of this proof. We now find an asymptotic expansion for 
\begin{align*}
J(\alpha,z):=\frac{z^{2\alpha}}{1+2\alpha}\,{}_2F_1\bigg({1,\frac{1}{2}+\alpha \atop \frac{3}{2}+\alpha}\;\bigg|\;z^2\bigg)+\frac{z^{-2\alpha}}{1-2\alpha}\,{}_2F_1\bigg({1,\frac{1}{2}-\alpha \atop \frac{3}{2}-\alpha}\;\bigg|\;z^2\bigg)-\frac{2}{z}\tanh^{-1}(z).    
\end{align*}
A straightforward asymptotic analysis reveals that, as $\alpha\rightarrow0$,
\begin{align}
\frac{z^{2\alpha}}{1+2\alpha}= 1+2(\ln(z)-1)\alpha +2(\ln^2(z)-2\ln(z)+2)\alpha^2+O(\alpha^3).  \label{gld1}
\end{align}
We now turn to the asymptotic analysis of the hypergeometric functions. We begin by noting that
\begin{align}
\frac{(\frac{1}{2}+\alpha)_n}{(\frac{3}{2}+\alpha)_n}=\frac{1+2\alpha}{2n+1+2\alpha}=\frac{1}{2n+1}+\frac{4n\alpha}{(2n+1)^2}-\frac{8n\alpha^2}{(2n+1)^3}+O(\alpha^3), \quad \alpha\rightarrow0. \label{gld2}
\end{align}
Therefore, as $\alpha\rightarrow0$,
\begin{align*}
{}_2F_1\bigg({1,\frac{1}{2}+\alpha \atop \frac{3}{2}+\alpha}\;\bigg|\;z^2\bigg)= S_0(z)+4S_1(z)\alpha-8S_2(z)\alpha^2+O(\alpha^3),    
\end{align*}
where
\begin{align*}
S_0(z)=\sum_{n=0}^\infty\frac{z^{2n}}{2n+1}=\frac{\tanh^{-1}(z)}{z}, \quad S_1(z)=\sum_{n=0}^\infty\frac{nz^{2n}}{(2n+1)^2}, \quad S_2(z)=\sum_{n=0}^\infty\frac{nz^{2n}}{(2n+1)^3}.    
\end{align*}
(Note that these infinite series are convergent under our assumption that $z\in\mathbb{C}$ is such that $z>0$, $z\not=1$ and $|z|\leq1$.) 
Putting this together we have that, as $\alpha\rightarrow0$,
\begin{align}
J(\alpha,z)&= \big(1+2(\ln(z)-1)\alpha +2(\ln^2(z)-2\ln(z)+2)\alpha^2\big)\big(S_0(z)+4S_1(z)\alpha-8S_2(z)\alpha^2\big)\nonumber\\
&\quad+\big(1-2(\ln(z)-1)\alpha +2(\ln^2(z)-2\ln(z)+2)\alpha^2\big)\big(S_0(z)-4S_1(z)\alpha-8S_2(z)\alpha^2\big)\nonumber\\
&\quad-\frac{2}{z}\tanh^{-1}(z)+O(\alpha^3)\nonumber\\
&=T(z)\alpha^2+O(\alpha^3), \label{wag2}
\end{align}
where
\begin{align*}
T(z)=4\big((\ln^2(z)-2\ln(z)+2)S_0(z)+4(\ln(z)-1)S_1(z))-4S_2(z)\big).    
\end{align*}
From the asymptotic expansions (\ref{wag1}) and (\ref{wag2}) we finally obtain that
\begin{align}\label{lim24}
\int_0^\infty xK_0^2(a x) K_0^2(b x) \,\mathrm{d}x=\frac{\pi^2}{8b^2}\lim_{\alpha\rightarrow0}\csc^2(\pi\alpha)J(\alpha,z)=\frac{\pi^2}{8b^2}\cdot\frac{1}{\pi^2}\cdot T(z)= \frac{1}{8b^2}T(z).  
\end{align}

To complete the derivation of the integral formula (\ref{li2}) it just remains to find closed-form formulas for the series $S_1(z)$ and $S_2(z)$. To this end, we note the formula
\begin{align*}
\sum_{n=0}^\infty\frac{z^{2n}}{(2n+1)^s}=\frac{1}{z}\sum_{n=0}^\infty\frac{z^{2n+1}}{(2n+1)^s}=\frac{1}{2z}\big[\mathrm{Li}_s(z)-\mathrm{Li}_s(-z)\big], \quad \mathrm{Re}\,s>1,\: |z|\leq1.    
\end{align*}
We then get that
\begin{align}\label{lim25}
S_1(z)&=\sum_{n=0}^\infty\frac{nz^{2n}}{(2n+1)^2}=\frac{1}{2}\bigg[\sum_{n=0}^\infty\frac{z^{2n}}{2n+1}-\sum_{n=0}^\infty\frac{z^{2n}}{(2n+1)^2}\bigg]\nonumber\\
&=\frac{1}{4z}\big[2\tanh^{-1}(z)-(\mathrm{Li}_2(z)-\mathrm{Li}_2(-z))]    
\end{align}
and, similarly,
\begin{align}\label{lim26}
S_2(z)&=\frac{1}{2}\bigg[\sum_{n=0}^\infty\frac{z^{2n}}{(2n+1)^2}-\sum_{n=0}^\infty\frac{z^{2n}}{(2n+1)^3}\bigg]\nonumber\\
&=\frac{1}{4z}\big[\mathrm{Li}_2(z)-\mathrm{Li}_2(-z)-(\mathrm{Li}_3(z)-\mathrm{Li}_3(-z))].   
\end{align}
On combining equations (\ref{lim24}), (\ref{lim25}) and (\ref{lim26}) and simplifying we obtain the integral formula (\ref{li2}).

Finally, we obtain the integral formula (\ref{fox1}) by taking the limit $a/b\rightarrow 1^{-}$ in the integral formula (\ref{li2}) and using the basic limit $\lim_{z\rightarrow1^{-}}\ln^2(z)\tanh^{-1}(z)=0$ and the standard formulas $\mathrm{Li}_3(1)=\zeta(3)$ and $\mathrm{Li}_3(-1)=-3\zeta(3)/4$.
\end{proof}

\begin{proof}[Proof of Theorem \ref{cor2.5}] 1. Applying equation (\ref{fgh}) to the integral formula (\ref{corf2}) (with $\alpha=\beta$ and $s=1$) gives that
\begin{align*}
&\int_0^\infty I_\alpha(a x)K_\alpha(a x) K_\alpha^2(b x) \,\mathrm{d}x \\
&\quad=\frac{1}{8a}G_{4,4}^{2,4}\bigg(\frac{a^2}{b^2}\;\bigg|\;{1+\alpha,1-\alpha,1,1 \atop \frac{1}{2}+\alpha,\frac{1}{2},\frac{1}{2}-\alpha,\frac{1}{2}}\bigg) \\
&\quad=\frac{1}{8a}\bigg\{\frac{\Gamma(\alpha)\Gamma^2(\frac{1}{2})\Gamma(\frac{1}{2}+\alpha)\Gamma(\frac{1}{2}-\alpha)}{\Gamma(\alpha+1)}\bigg(\frac{a}{b}\bigg)\, {}_4F_3\bigg({\frac{1}{2},\frac{1}{2},\alpha+\frac{1}{2},\frac{1}{2}-\alpha \atop 1,\alpha+1,1-\alpha}\;\bigg|\;\frac{a^2}{b^2}\bigg)\nonumber\\
&\quad\quad+\frac{\Gamma(-\alpha)\Gamma(\frac{1}{2})\Gamma^2(\alpha+\frac{1}{2})\Gamma(2\alpha+\frac{1}{2})}{\Gamma(\alpha+1)\Gamma(2\alpha+1)}\bigg(\frac{a}{b}\bigg)^{2\alpha+1} {}_4F_3\bigg({\frac{1}{2},\alpha+\frac{1}{2},\alpha+\frac{1}{2},2\alpha+\frac{1}{2} \atop \alpha+1,\alpha+1,2\alpha+1}\;\bigg|\;\frac{a^2}{b^2}\bigg)\bigg\}.  
\end{align*}
We obtain the integral formula (\ref{sch1}) by simplifying the ratios of gamma functions in the above expression. We simplify the first ratio by using the standard formulas $\Gamma(1/2)=\sqrt{\pi}$ and $\Gamma(z+1)=z\Gamma(z)$ and the identity (\ref{refl}), whilst we simplify the second ratio by using the special value $\Gamma(1/2)=\sqrt{\pi}$, the reflection formula (\ref{ref0}) and the duplication formula (\ref{dupfor}).

\vspace{3mm}

\noindent 2. From the integral formula (\ref{corf2}) (with $\alpha=\beta$ and $s=2$) we get that
\begin{align}
\int_0^\infty x I_\alpha(ax)K_\alpha(ax)K_\alpha^2(bx)\,\mathrm{d}x&=\frac{1}{8a^{2}}G_{4,4}^{2,4}\bigg(\frac{a^2}{b^2}\;\bigg|\;{1+\alpha,1-\alpha,1,\frac{3}{2} \atop 1+\alpha,1,1-\alpha,\frac{1}{2}}\bigg)\nonumber\\
&=\frac{1}{8a^{2}}G_{3,3}^{2,3}\bigg(\frac{a^2}{b^2}\;\bigg|\;{1+\alpha,1,\frac{3}{2} \atop 1+\alpha,1,\frac{1}{2}}\bigg), \label{tanhh}
\end{align}
where we reduced the order of the $G$-function in the second step using equation (\ref{lukeformula0}). We obtain the integral formula (\ref{forab2}) by applying the reduction formula (\ref{fgh}) to the $G$-function in equation (\ref{tanhh}). This is carried out similarly to in part 1 of this proof, so we omit the details. However, we do note that due to repeated factors the ${}_3F_2$ functions reduce to ${}_2F_1$ functions by formula (\ref{fred}), an example of which is given by
\begin{align*}
{}_3F_2\bigg({1-\alpha,1,\frac{1}{2} \atop 1-\alpha,\frac{3}{2}}\;\bigg|\;\frac{a^2}{b^2}\bigg)&={}_2F_1\bigg({1,\frac{1}{2} \atop \frac{3}{2}}\;\bigg|\;\frac{a^2}{b^2}\bigg)=\frac{b}{a}\tanh^{-1}\bigg(\frac{a}{b}\bigg),   
\end{align*}
where we used the reduction formula (\ref{2f1redtanh}) in the second step.

We obtain the integral formula (\ref{foraa2}) by taking the limit $a/b\rightarrow1^{-}$ in in the integral formula (\ref{forab2}) by applying the asymptotic expansion (\ref{2f1exp}). This procedure is very similar to the derivation of the integral formula (\ref{forab}) from the integral formula (\ref{foraa}) that was given in the proof of Theorem \ref{cor1.3}, and we therefore omit the details. 

We now derive the integral formula (\ref{lili}). In addition to our assumption that $\mathrm{Re}\,a,\mathrm{Re}\,b>0$ we now also suppose that $a\not=b$ and $|a/b|\leq1$.
We will derive the integral formula (\ref{lili}) by taking the limit $\alpha\rightarrow0$ in the integral formula (\ref{forab2}). Recall the notation $S_0(z)$ and $S_1(z)$ from the proof of Theorem \ref{cor1.3} and let $z=a/b$. Then, by applying the asymptotic expansions (\ref{gld1}) and (\ref{gld2}) we obtain that, as $\alpha\rightarrow0$,
\begin{align}
\mathcal{J}(\alpha,z)&:=\frac{\tanh^{-1}(z)}{z}-\frac{z^{2\alpha}}{1+2\alpha}\,{}_2F_1\bigg({1,\frac{1}{2}+\alpha \atop \frac{3}{2}+\alpha}\;\bigg|\;z^2\bigg)\nonumber\\
&=\frac{\tanh^{-1}(z)}{z}-\big(1+2(\ln(z)-1)\alpha\big)\big(S_0(z)+4S_1(z)\alpha\big)+O(\alpha^2)\nonumber\\
&=2\big((1-\ln(z))S_0(z)-2S_1(z)\big)\alpha+O(\alpha^2).\label{gld3}
\end{align}
We also have that
\begin{align}
\csc(\pi\alpha)=\frac{1}{\pi\alpha}+O(\alpha), \quad \alpha\rightarrow0. \label{gld4}   
\end{align}
By applying the asymptotic expansions (\ref{gld3}) and (\ref{gld4}) to the integral formula (\ref{forab2}) we obtain that
\begin{align*}
\int_0^\infty  x I_0(ax)K_0(ax)K_0^2(bx) \,\mathrm{d}x&=\frac{\pi}{4b^2}\lim_{\alpha\rightarrow0}\csc(\pi\alpha)\mathcal{J}(\alpha,z)\\
&=\frac{\pi}{4b^2}\cdot\frac{1}{\pi}\cdot2\big((1-\ln(z))S_0(z)-2S_1(z)\big)\\
&=\frac{1}{2b^2}\big((1-\ln(z))S_0(z)-2S_1(z)\big),
\end{align*}
and on using that $S_0(z)=z^{-1}\tanh^{-1}(z)$ and the formula (\ref{lim25}) for the sum $S_1(z)$ we obtain the desired integral formula (\ref{lili}).

We now derive the integral formula (\ref{lilix}), which we do so by adapting the derivation of the integral formula (\ref{lili}). We now suppose that $a\not=b$ and $|a/b|\geq1$, that is $z\not=1$ and $|z|\geq1$ (along with the standing assumption that $\mathrm{Re}\,z>0$). We begin by noting that from identity (\ref{apack}) we have that
\begin{align*}
{}_2F_1\bigg({1,\frac{1}{2}+\alpha \atop \frac{3}{2}+\alpha}\;\bigg|\;z^2\bigg)&=\frac{\Gamma(\frac{3}{2}+\alpha)\Gamma(\alpha-\frac{1}{2})}{\Gamma^2(\frac{1}{2}+\alpha)}(-z^2)^{-1}\,{}_2F_1\bigg({1,\frac{1}{2}-\alpha \atop \frac{3}{2}-\alpha}\;\bigg|\;\frac{1}{z^2}\bigg) \\
&\quad+\Gamma(\tfrac{3}{2}+\alpha)\Gamma(\tfrac{1}{2}-\alpha)(-z^2)^{-1/2-\alpha}\\
&=\frac{1+2\alpha}{1-2\alpha}z^{-2}\,{}_2F_1\bigg({1,\frac{1}{2}-\alpha \atop \frac{3}{2}-\alpha}\;\bigg|\;\frac{1}{z^2}\bigg)+\frac{\pi(1+2\alpha)}{2\cos(\pi\alpha)}\mathrm{e}^{-\mathrm{i}\pi(1/2+\alpha)}z^{-1-2\alpha},
\end{align*}
where we used our usual arguments to simplify the ratio and product of gamma functions. Since
\begin{align*}
\mathrm{e}^{-\mathrm{i}\pi(1/2+\alpha)}=\mathrm{e}^{-\mathrm{i}\pi/2}\mathrm{e}^{-\mathrm{i}\pi\alpha}=-\mathrm{i}\mathrm{e}^{-\mathrm{i}\pi\alpha}=-\mathrm{i}\cos(\pi\alpha)-\sin(\pi\alpha),  
\end{align*}
we have that
\begin{align*}
\mathcal{J}(\alpha,z)&=\frac{1}{z^2}\bigg\{z\bigg(\tanh^{-1}(z)+\frac{\mathrm{i}\pi}{2}+\frac{\pi}{2}\tan(\pi\alpha)\bigg)-\frac{z^{2\alpha}}{1-2\alpha}\,{}_2F_1\bigg({1,\frac{1}{2}-\alpha \atop \frac{3}{2}-\alpha}\;\bigg|\;\frac{1}{z^2}\bigg)\bigg\}   \\
&=\frac{1}{z^2}\bigg\{z\bigg(\tanh^{-1}\bigg(\frac{1}{z}\bigg)+\frac{\pi}{2}\tan(\pi\alpha)\bigg)-\frac{z^{2\alpha}}{1-2\alpha}\,{}_2F_1\bigg({1,\frac{1}{2}-\alpha \atop \frac{3}{2}-\alpha}\;\bigg|\;\frac{1}{z^2}\bigg)\bigg\}.
\end{align*}
Let $w=1/z$. Then arguing as before we have that, as $\alpha\rightarrow0$,
\begin{align}
\mathcal{J}(\alpha,z)&=w^2\bigg\{\frac{\tanh^{-1}(w)}{w}+\frac{\pi\tan(\pi\alpha)}{2w}-(1+2\alpha)\big(1-2\ln(w)\alpha\big)\big(S_0(w)-4S_1(w)\alpha\big)\bigg\}+O(\alpha^2)\nonumber\\
&=w^2\bigg\{\frac{\pi^2}{2w}+2\big(2S_1(w)+(\ln(w)-1)S_0(w)\big)\bigg\}\alpha+O(\alpha^2).   \label{gld5}
\end{align}
(Note that $\mathrm{Re}\,w>0$, $|w|\leq1$ and $w\not=1$, which ensures that the series $S_0(w)$ and $S_1(w)$ are convergent.)
Applying the asymptotic expansions (\ref{gld4}) and (\ref{gld5}) to the integral formula (\ref{forab2}) now yields
\begin{align*}
\int_0^\infty  x I_0(ax)K_0(ax)K_0^2(bx) \,\mathrm{d}x&=\frac{\pi}{4b^2}\lim_{\alpha\rightarrow0}\csc(\pi\alpha)\mathcal{J}(\alpha,z)\\
&=\frac{\pi}{4b^2}\cdot\frac{1}{\pi}\cdot \bigg(\frac{\pi^2w}{2}+2w^2\big(2S_1(w)+(\ln(w)-1)S_0(w)\big)\bigg)\\
&=\frac{\pi^2}{8ab}+\frac{1}{2a^2}\big(2S_1(w)+(\ln(w)-1)S_0(w)\big),
\end{align*}
and by using that $S_0(z)=z^{-1}\tanh^{-1}(z)$ and the formula (\ref{lim25}) for the sum $S_1(z)$ we get the integral formula (\ref{lilix}).

Finally, we deduce the integral formula (\ref{fox2}) by taking the limit $a/b\rightarrow 1^{-}$ in the integral formula (\ref{lili}) and using the limit $\lim_{z\rightarrow1^{-}}\ln(z)\tanh^{-1}(z)=0$ and the standard special values of the dilogarithm $\mathrm{Li}_2(1)=\pi^2/6$ and $\mathrm{Li}_2(-1)=-\pi^2/12$.
\end{proof}

\begin{proof}[Proof of Theorem \ref{cor2.6}] 1. Setting $\alpha=\beta$ and $s=1$ in the integral formula (\ref{corf3}) gives that
\begin{align*}
\int_0^\infty I_\alpha^2(a x) K_\alpha^2(b x) \,\mathrm{d}x
&= \frac{1}{4a^2}\bigg(\frac{a}{b}\bigg)^{2\alpha+1}\frac{\Gamma^2(\alpha+\frac{1}{2})\Gamma(2\alpha+1)\Gamma(\frac{1}{2})}{\Gamma^2(\alpha+1)\Gamma(2\alpha+1)}\\
&\quad\times {}_4F_3\bigg({\alpha+\frac{1}{2},2\alpha+\frac{1}{2},\alpha+\frac{1}{2},\frac{1}{2} \atop \alpha+1,2\alpha+1,\alpha+1}\;\bigg|\;\frac{a^2}{b^2}\bigg),
\end{align*}
and simplifying the ratio of gamma functions using the duplication formula (\ref{dupfor}) and the standard formula $\Gamma(1/2)=\sqrt{\pi}$ yields the desired formula (\ref{haw}).

\vspace{3mm}

\noindent 2. We now set $\alpha=\beta$ and $s=2$ in the integral formula (\ref{corf3}) and apply the reduction formula (\ref{fred}) to obtain that
\begin{align*}
\int_0^\infty xI_\alpha^2(a x) K_\alpha^2(b x) \,\mathrm{d}x&=\frac{1}{4a^2}\bigg(\frac{a}{b}\bigg)^{2\alpha+2}\frac{\Gamma(\alpha+\frac{1}{2})}{\Gamma(\alpha+\frac{3}{2})}\,{}_4F_3\bigg({\alpha+\frac{1}{2},2\alpha+1,\alpha+1,1 \atop \alpha+1,2\alpha+1,\alpha+\frac{3}{2}}\;\bigg|\;\frac{a^2}{b^2}\bigg)\\
&=\frac{(a/b)^{2\alpha}}{2b^2(1+2\alpha)}\,{}_2F_1\bigg({\alpha+\frac{1}{2},1 \atop \alpha+\frac{3}{2}}\;\bigg| \;\frac{a^2}{b^2}\bigg),   
\end{align*}
as required.
\end{proof}


\appendix
\section{
Special functions
}\label{appa}

In this appendix, we define the modified Bessel functions of the first and second kind, the generalized hypergeometric function, the Lauricella function $F_C$ and the Meijer $G$-function, and also state some of their relevant basic properties. Unless otherwise stated, all these properties can be found in the standard references \cite{luke,olver}.

\subsection{Modified Bessel functions}\label{appa1}

The \emph{modified Bessel function of the first kind} is defined, for $\nu\in\mathbb{C}$ and $z\in\mathbb{C}$, by the power series
\begin{equation}\label{idef}
I_\nu(z)=\sum_{k=0}^\infty \frac{1}{k!\Gamma(k+\nu+1)}\bigg(\frac{z}{2}\bigg)^{2k+\nu}.    
\end{equation}
The \emph{modified Bessel function of the first kind} can be defined, for $\nu\in\mathbb{C}$ and $\mathrm{Re}\,z>0$, by the integral
\begin{equation*}
K_\nu(z)=\int_0^\infty\mathrm{e}^{-z\cosh t}\cosh(\nu t)\,\mathrm{d}t.   
\end{equation*}

For $n\in\mathbb{Z}$, $\nu\in\mathbb{C}$ and $z\in\mathbb{C}$, the following identities hold:
\begin{align} \label{ipar} I_{-n}(z)&=I_n(z), \\
\label{par} K_{-\nu}(z)&=K_{\nu}(z),
\end{align}
and the relation (\ref{ipar}) generalises to
\begin{align}
I_{-\nu}(z)=I_\nu(z)+\frac{2\sin(\pi\nu)}{\pi}K_\nu(z), \quad \nu\in\mathbb{C}. \label{ikrel}   
\end{align}
Let $\mathrm{Re}\,z>0$ and fix $\nu\in\mathbb{C}$. Then the functions $I_\nu(z)$ and $K_\nu(z)$ possess the following asymptotic behaviour:
\begin{align}\label{Itend0}I_{\nu} (z) &\sim \frac{(\frac{1}{2}z)^\nu}{\Gamma(\nu +1)}, 
\quad z \rightarrow 0, \: \nu\notin\{-1,-2,-3,\ldots\}, \\
\label{Ktend0}K_{\nu} (z) &\sim 2^{\nu -1} \Gamma (\nu) z^{-\nu},  \quad z \rightarrow 0, \quad \mathrm{Re}\,\nu>0,\\
\label{Ktend00}K_0(z)&\sim -\ln(z), \quad z\rightarrow0, \\
\label{Itendinfinity}I_{\nu} (z) &\sim \frac{\mathrm{e}^{z}}{\sqrt{2\pi z}}, \quad z \rightarrow\infty, \\
\label{Ktendinfinity} K_{\nu} (z) &\sim \sqrt{\frac{\pi}{2z}} \mathrm{e}^{-z}, \quad z \rightarrow \infty.
\end{align}
When $\nu\in\{-1,-2,-3,\ldots\}$, we can use the relation (\ref{ipar}) to establish the limiting behaviour of $I_\nu(z)$ as $z\rightarrow0$, and when $\mathrm{Re}\,\nu<0$ we can use the relation (\ref{par}) to infer the limiting behaviour of $K_\nu(z)$ as $z\rightarrow0$. If $\nu\not=0$ is purely imaginary then $K_\nu(z)=O(1)$ as $z\rightarrow0$ (see \cite[equations 10.45.2 and 10.45.7]{olver}).

\subsection{Generalized hypergeometric and related functions}

The \emph{generalized hypergeometric function} is defined by the power series
\begin{equation}
\label{gauss}
{}_pF_q\bigg({a_1,\ldots,a_p \atop b_1,\ldots,b_q} \; \bigg| \;z\bigg)=\sum_{k=0}^\infty\frac{(a_1)_k\cdots(a_p)_k}{(b_1)_k\cdots(b_q)_k}\frac{z^k}{k!},
\end{equation}
for $|z|<1$, and by analytic continuation elsewhere. The \emph{Pochhammer symbol} is given by $(v)_k=v(v+1)\cdots(v+k-1)$. The function ${}_2F_1(a,b;c;z)
$ is called the \emph{(Gaussian) hypergeometric function}. The \emph{regularized generalized hypergeometric function} is defined by
\begin{equation}\label{gauss2}
{}_p\tilde{F}_q\bigg({a_1,\ldots,a_p \atop b_1,\ldots,b_q} \; \bigg| \;z\bigg)=\frac{1}{\Gamma(b_1)\cdots\Gamma(b_q)}\, {}_pF_q\bigg({a_1,\ldots,a_p \atop b_1,\ldots,b_q} \; \bigg| \;z\bigg).    
\end{equation}
It is immediate from the series representation (\ref{gauss}) that the generalized hypergeometric function and the regularized generalized hypergeometric function reduce to ones of lower order if one of the $a_h$'s, $h=1,\ldots,p$, is equal to one of the $b_j$'s, $j=1,\ldots,q$. For example, for $p,q\geq1$,
\begin{align}\label{fred}
{}_pF_q\bigg({a_1,\ldots,a_p \atop b_1,\ldots,b_{q-1},a_{p}} \; \bigg| \;z\bigg)&={}_{p-1}F_{q-1}\bigg({a_1,\ldots,a_{p-1} \atop b_1,\ldots,b_{q-1}} \; \bigg| \;z\bigg), \\
\label{fred2}
{}_p\tilde{F}_q\bigg({a_1,\ldots,a_p \atop b_1,\ldots,b_{q-1},a_{p}} \; \bigg| \;z\bigg)&=\frac{1}{\Gamma(a_p)}\,{}_{p-1}\tilde{F}_{q-1}\bigg({a_1,\ldots,a_{p-1} \atop b_1,\ldots,b_{q-1}} \; \bigg| \;z\bigg).
\end{align}

The \emph{Lauricella function} $F_C$ is defined by
\begin{align}
\displaystyle
F_C^{(n)} ( a, b; c_1, \ldots, c_n; z_1, \ldots, z_n )  =
\sum_{k_1= 0}^\infty \cdots \sum_{k_n= 0}^\infty
\frac {(a)_{k_1 + \cdots + k_n} (b)_{k_1 + \cdots + k_n}}
{( c_1 )_{k_1} \cdots ( c_n )_{k_n}}
\frac {z_1^{k_1} \cdots z_n^{k_n}}
{k_1! \cdots k_n!},
\label{hyp}
\end{align}
for $|z_1|^{1/2}+\cdots+|z_n|^{1/2}<1$.

The hypergeometric function satisfies the identity
\begin{align}
{}_2F_1\bigg({a,b \atop c}\;\bigg|\;z\bigg)&=\frac{\Gamma(c)\Gamma(b-a)}{\Gamma(b)\Gamma(c-a)}(-z)^{-a}\,{}_2F_1\bigg({a,a-c+1 \atop a-b+1}\;\bigg|\;\frac{1}{z}\bigg)\nonumber\\
&\quad+\frac{\Gamma(c)\Gamma(a-b)}{\Gamma(a)\Gamma(c-b)}(-z)^{-b}\,{}_2F_1\bigg({b,-c+1 \atop b-a+1}\;\bigg|\;\frac{1}{z}\bigg)
. \label{apack}
\end{align}
If $\mathrm{Re}(c-a-b)>0$, then
\begin{align}\label{2f11}
{}_2F_1\bigg({a,b \atop c} \; \bigg| \; 1\bigg)=\frac{\Gamma(c)\Gamma(c-a-b)}{\Gamma(c-a)\Gamma(c-b)}.
\end{align}
The following asymptotic expansion holds for the hypergeometric function. As $z\rightarrow 1^-$,
\begin{equation}
{}_2F_1\bigg({a,b \atop a+b} \; \bigg| \;z\bigg)=-\frac{\Gamma(a+b)}{\Gamma(a)\Gamma(b)}\ln(1-z)-\frac{\Gamma(a+b)}{\Gamma(a)\Gamma(b)}\big(\psi(a)+\psi(b)+2\gamma\big)+o(1) \label{2f1exp}   
\end{equation}
(see http://functions.wolfram.com/07.23.06.0014.01).

For $|z|<1$ and $a\in\mathbb{N}$ a non-negative integer the hypergeometric function takes an elementary form:
\begin{align}
{}_2F_1\bigg({a+\frac{1}{2},1 \atop a+\frac{3}{2}} \; \bigg| \; z\bigg)&=\sum_{k=0}^\infty\frac{(a+\frac{1}{2})_k}{(a+\frac{3}{2})_k}z^k=\sum_{k=0}^\infty\frac{2a+1}{2k+2a+1}z^k \nonumber \\
&=\frac{2a+1}{z^{a}}\bigg(\frac{\tanh^{-1}(\sqrt{z})}{\sqrt{z}}-\sum_{k=0}^{a-1}\frac{z^k}{2k+1}\bigg), \label{2f1redtanh} \end{align}
where the final equality follows since $\sum_{k=0}^\infty z^{k}/(2k+1)=z^{-1/2}\tanh^{-1}(\sqrt{z})$. Another elementary form is given by the following. For $|z|<1$ and $n\in\mathbb{Z}^+$,
\begin{align}
{}_2F_1\bigg({n,1 \atop n+1} \; \bigg| \; z\bigg)&=\sum_{k=0}^\infty \frac{(n)_k}{(n+1)_k}z^k=\sum_{k=0}^\infty\frac{nz^k}{n+k}=\frac{n}{z^n}\bigg(-\ln(1-z)-\sum_{k=1}^{n-1}\frac{z^k}{k}\bigg). \label{flog}
\end{align}
Further special cases of the hypergeometric function are given by
\begin{align}
{}_2F_1\bigg({\frac{1}{6},1 \atop \frac{7}{6}} \; \bigg| \; z\bigg)&=\frac{1}{12z^{1/6}}\big(4\tanh^{-1}(z^{1/6})+2\sqrt{3}\tan^{-1}(1-z^{1/3},\sqrt{3}z^{1/6})\nonumber\\
&\quad+\ln(z^{1/3}+z^{1/6}+1)-\ln(z^{1/3}-z^{1/6}+1)\big),\label{red11} \\
{}_2F_1\bigg({\frac{5}{6},1 \atop \frac{11}{6}} \; \bigg| \; z\bigg)&=\frac{5}{12z^{5/6}}\big(4\tanh^{-1}(z^{1/6})-2\sqrt{3}\tan^{-1}(1-z^{1/3},\sqrt{3}z^{1/6})\nonumber\\
&\quad+\ln(z^{1/3}+z^{1/6}+1)-\ln(z^{1/3}-z^{1/6}+1)\big), \label{red22}\\
{}_2F_1\bigg({\frac{1}{4},1 \atop \frac{5}{4}} \; \bigg| \; z\bigg)&=\frac{1}{2z^{1/4}}\big(\tanh^{-1}(z^{1/4})+\tan^{-1}(z^{1/4})\big), \label{red33} \\
{}_2F_1\bigg({\frac{3}{4},1 \atop \frac{7}{4}} \; \bigg| \; z\bigg)&=\frac{3}{2z^{3/4}}\big(\tanh^{-1}(z^{1/4})-\tan^{-1}(z^{1/4})\big); \label{red44}
\end{align}
see http://functions.wolfram.com/07.23.03.0299.01, http://functions.wolfram.com/07.23.03.0387.01, http://functions.wolfram.com/07.23.03.ai1i.01 and http://functions.wolfram.com/07.23.03.aitp.01, respectively. Here, the \emph{two-argument inverse tangent} is given, for $x,y\in\mathbb{C}$, by
\begin{align*}
\tan^{-1}(x,y)=-\mathrm{i}\ln\bigg(\frac{x+\mathrm{i}y}{\sqrt{x^2+y^2}}\bigg)   \end{align*}
(see http://functions.wolfram.com/01.15.02.0001.01).

\subsection{The Meijer $G$-function}

The \emph{Meijer $G$-function} is defined, for $z\in\mathbb{C}$, by the contour integral
\begin{equation}\label{mdef}G^{m,n}_{p,q}\bigg(z \, \bigg|\, {a_1,\ldots, a_p \atop b_1,\ldots,b_q} \bigg)=\frac{1}{2\pi \mathrm{i}}\int_L\frac{\prod_{j=1}^m\Gamma(b_j-r)\prod_{j=1}^n\Gamma(1-a_j+r)}{\prod_{j=n+1}^p\Gamma(a_j-r)\prod_{j=m+1}^q\Gamma(1-b_j+r)}z^r\,\mathrm{d}r,
\end{equation}
where the integration path $L$ separates the poles of the factors $\Gamma(b_j-s)$ from those of the factors $\Gamma(1-a_j+s)$.  We use the convention that the empty product is $1$.

The Meijer $G$-function reduces to one of lower order if one of the $a_h$'s, $h=1,\ldots,n$, is equal to one of the $b_j$'s, $j=m+1,\ldots,q$, or if one of the $a_h$'s, $h=n+1,\ldots,p$, is equal to one of the $b_j$'s, $j=1,\ldots,m$. For example, 
\begin{align}
\label{lukeformula0}G_{p,q}^{m,n}\bigg(z \; \bigg| \;{a_1,\ldots,a_{p} \atop b_1,\ldots,b_{q-1},a_1}\bigg)&=G_{p-1,q-1}^{m,n-1}\bigg(z \; \bigg| \;{a_2,\ldots,a_{p} \atop b_1,\ldots,b_{q-1}}\bigg), \quad n,p,q\geq 1, \\
\label{lukeformula}G_{p,q}^{m,n}\bigg(z \; \bigg| \;{a_1,\ldots,a_{p-1},b_1 \atop b_1,\ldots,b_q}\bigg)&=G_{p-1,q-1}^{m-1,n}\bigg(z \; \bigg| \;{a_1,\ldots,a_{p-1} \atop b_2,\ldots,b_q}\bigg), \quad m,p,q\geq 1.     
\end{align}
Provided certain conditions are met, the $G$-function may be expressed as a finite sum of generalized hypergeometric functions. One example is given by equation 16.17.2 of \cite{olver}, which we now recall. Suppose that $p\leq q$ and that none of the $b_h$ terms, $h=1,\ldots,m$, differ by an integer or zero, and all poles are simple. Then
\begin{align}
G^{m,n}_{p,q}\bigg(z \, \bigg|\, {a_1,\ldots, a_p \atop b_1,\ldots,b_q} \bigg)&=\sum_{h=1}^m A_{p,q,h}^{m,n}(z)\, {}_pF_{q-1}\bigg({1+b_h-a_1,\ldots,1+b_h-a_p \atop 1+b_h-b_1,\ldots*\ldots,1+b_h-b_q} \; \bigg| \;(-1)^{p-m-n}z\bigg),\label{fgh}
\end{align}
where
\begin{align*}
A_{p,q,h}^{m,n}(z)= \frac{\prod_{j=1, \,j\not=h}^m\Gamma(b_j-b_h)\prod_{j=1}^n\Gamma(1+b_h-a_j)z^{b_h}}{\prod_{j=m+1}^q\Gamma(1+b_h-b_j)\prod_{j=n+1}^p\Gamma(a_j-b_h)}.   
\end{align*}
Here, the notation $*$ means that the entry $1+b_h-b_h$ is omitted.

\section*{Acknowledgements}
The author is funded in part by EPSRC grant EP/Y008650/1.

\end{document}